\documentclass[reqno]{amsart}
\usepackage{amssymb, amsmath}
\usepackage{array}
\usepackage{xcolor}

\usepackage{comment}

\usepackage[noadjust]{cite}
\usepackage[hidelinks]{hyperref}

\usepackage{fullpage}

\makeatletter
\def\thmhead@plain#1#2#3{%
  \thmname{#1}\thmnumber{\@ifnotempty{#1}{ }\@upn{#2}}%
  \thmnote{ {\the\thm@notefont#3}}}
\let\thmhead\thmhead@plain
\makeatother

\newtheorem{theorem}{Theorem}
\newtheorem{corollary}[theorem]{Corollary}
\newtheorem{lemma}[theorem]{Lemma}
\newtheorem{proposition}[theorem]{Proposition}

\theoremstyle{definition}

\newtheorem{definition}[theorem]{Definition}
\newtheorem{remark}[theorem]{Remark}
\newtheorem{example}[theorem]{Example}

\numberwithin{equation}{section}
\numberwithin{theorem}{section}

\newcommand{\N}{\mathbb{N}}

\DeclareMathOperator{\sgn}{sgn}
\DeclareMathOperator{\I}{Id}
\DeclareMathOperator{\Id}{Id}
\DeclareMathOperator{\End}{End}
\DeclareMathOperator{\spa}{span}
\DeclareMathOperator{\cha}{char}

\DeclareMathOperator{\V}{var}
\newcommand{\M}{\mathcal{M}}
\newcommand{\IM}{\mathcal{I}\mathcal{M}}
\newcommand{\dia}{\symbol{35}}
\DeclareMathOperator{\op}{op}

\DeclareMathOperator{\ide}{id}

\begin{document}

\title{$W$-algebras with involution: \\ polynomial identities and asymptotic growth}

\author[G.~Giordani]{Ginevra Giordani}
\address{Ginevra Giordani, DISIM, University of L'Aquila, 67100
L'Aquila, Italy}
\email{ginevra.giordani@graduate.univaq.it}


\author[A.~Ioppolo]{Antonio Ioppolo}
\address{Antonio Ioppolo, DISIM, University of L'Aquila, 67100
L'Aquila, Italy}
\email{antonio.ioppolo@univaq.it}


\author[C.~Rizzo]{Carla Rizzo}
\address{Carla Rizzo, CMUC, Department of Mathematics, University of Coimbra, 3000-143
Coimbra, Portugal.}
\email{carlarizzo@mat.uc.pt}

\thanks{A. Ioppolo was supported by “Progetti di Ateneo 2025” of the University of L'Aquila and by GNSAGA-INDAM, Italy.}

\thanks{C. Rizzo was financially supported by the Fundação para a Ciência e a Tecnologia (Portuguese Foundation for Science and Technology) under the scope of the projects UID/00324/2025 (https://doi.org/10.54499/UID/00324/2025) (Centre for Mathematics of the University of Coimbra).}





\begin{abstract}
Let $W$ be an algebra with involution over a field of characteristic zero. We develop a theory of $W$-polynomial identities with involution for finite-dimensional $(W,*)$-algebras $A$, using the multiplier algebra with involution of $A$. We prove that the $(W,*)$-exponent of $A$ exists and coincides with the ordinary $*$-exponent. We then consider a four-dimensional subalgebra $M$ of the algebra of $4\times4$ upper triangular matrices, endowed with the reflection involution, and study two non-equivalent $(W,*)$-algebra structures on it. For both structures, we determine the corresponding $T_W^*$-ideals of $W$-polynomial identities with involution and compute the $(W,*)$-codimension sequences explicitly; for one of them, we also determine the complete $(W,*)$-cocharacter sequence. Finally, we prove that these two $W$-algebras with involution generate distinct varieties of almost polynomial growth.
\end{abstract}

\maketitle

\section{Introduction}

Polynomial identities provide a powerful tool for studying algebras by means of the polynomial relations that hold under all evaluations. 
A natural extension is obtained by allowing the coefficients of the polynomials to belong to a fixed associative algebra $W$, rather than being restricted to the ground field. In this setting, elements of $W$ may occur at arbitrary positions among the variables, reflecting their role as noncommutative coefficients. This leads to the notion of a \emph{$W$-polynomial}.

The corresponding notion of a \emph{$W$-polynomial identity} naturally extends the concept of a generalized polynomial identity, introduced by Amitsur in 1965 \cite{AmiGPI} and subsequently developed by several authors (see \cite{Row} and the references therein). For unital algebras, the two notions coincide, and $W$-polynomial identities are therefore often referred to as generalized polynomial identities.

The natural framework for $W$-polynomial identities is provided by the theory of $W$-algebras. Roughly speaking, a $W$-algebra is an associative algebra $A$ endowed with a $W$-bimodule structure satisfying some additional $W$-associative conditions. This theory has been recently developed, proving to be particularly useful in the study of codimension growth and varieties of algebras with $W$-polynomial identities (see \cite{MarRiz,MarRiz2}).

The theory of generalized polynomial identities has a classical counterpart for algebras with involution. Following Amitsur's work on identities in rings with involution \cite{Amitsur1969}, Martindale studied generalized polynomial identities for prime rings with involution \cite{Martindale1972}. Rowen subsequently developed a structural framework for generalized identities in rings with involution \cite{RowGI1,RowGI2,RowGI3}, which was later incorporated into the systematic treatment of polynomial and generalized identities in \cite{Row}.
In the present paper, we extend generalized polynomial identities with involution to the broader setting of $W$-algebras, focusing in particular on their codimensions, cocharacters, and asymptotic growth.

A fundamental role in our approach is played by multiplier algebras, which provide the natural setting for incorporating the action of a possibly non-unital algebra $W$ on an associative algebra. The notion of a multiplier algebra of an associative algebra dates back to Hochschild \cite{Hochschild} and was later developed in the context of $C^*$-algebras by Busby \cite{Busby}; it has since found applications in various areas, including partial dynamical systems \cite{DokEx,Exel,DokRioSim}.
In our setting, an action of $W$ on $A$ is encoded by a homomorphism $\Phi$ from $W$ into the multiplier algebra $\M(A)$ of $A$. When $W$ and $A$ are endowed with involutions, $\M(A)$ naturally inherits an involution, and $\Phi$ is a homomorphism of algebras with involution. This provides the structural framework for our study of $W$-algebras with involution.

A central problem in PI-theory is the asymptotic behavior of codimension sequences, which quantify the multilinear polynomial identities of an algebra. In particular, their exponential growth rate is captured by the PI-exponent in the ordinary setting \cite{GiaZai98,GiaZai99}. Analogous growth rates have been studied extensively for algebras endowed with additional structures; see, for instance, \cite{AljadeffGiambruno2013,AljadeffGiambrunoLaMattina2011,DosSantosExp,GiaIopLaMTraceExp,GiambrunoLaMattina2010,GiaZai,Gordienko2013,IopPseudoExp,IopSuperExp}.
In the framework of $W$-algebras, the exponential growth of $W$-codimensions has recently been investigated for finite-dimensional unital $W$-algebras \cite{MarRiz}, finite-dimensional $W$-algebras not necessarily with unity \cite{RizSurvey}, and finite-dimensional group-graded $W$-algebras \cite{BusMartRiz}. For algebras with involution, the existence and integrality of the $*$-exponent were established in \cite{GiaPolVal}.

It is worth noting that the exponent associated with an additional algebraic structure need not coincide with the ordinary PI-exponent; this is already the case for algebras with involution \cite{GiaMis2001}. On the other hand, for finite-dimensional algebras, the two exponents are known to coincide in some settings, including $W$-algebras, group-graded $W$-algebras, and algebras endowed with an action of a Lie algebra by derivations \cite{Riz2023}.

One of the main objectives of this paper is to show that, for a finite-dimensional $W$-algebra $A$ with involution, the corresponding $(W,*)$-exponent $\exp^{W,*}(A)$ exists and coincides with the $*$-exponent of $A$. In particular, it is a non-negative integer. Thus, although the additional $W$-action may substantially affect the $W$-polynomial identities with involution and their codimension sequence, it does not alter their exponential rate of growth.

Since $\exp^{W,*}(A)$ is an integer, it is natural to investigate the boundary between polynomial and exponential growth, corresponding to the cases $\exp^{W,*}(A)=1$ and $\exp^{W,*}(A)\geq 2$. To this end, we focus on the algebra with involution $M$, a subalgebra of the upper triangular matrices of order four endowed with the reflection involution introduced in \cite{MisVal}.

We determine the only two $W$-algebra with involution structures on $M$, denoted by $M^F$ and $M^D$, that generate distinct varieties of $W$-algebras with involution of almost polynomial growth. More precisely, $\exp^{W,*}(M^F)=\exp^{W,*}(M^D)=2$, whereas every proper subvariety has polynomially bounded codimension sequence (equivalently $(W,*)$-exponent $1$). For both algebras, we determine generators of the corresponding $T_W^*$-ideals of $W$-polynomial identities with involution and compute their codimensions explicitly. We also give a complete description of the generalized $*$-cocharacters of $M^D$.

The paper is organized as follows. Section~2 describes the structure of $W$-algebras with involution in terms of multiplier algebras. In Section~3, we introduce $W$-polynomial identities with involution, together with the corresponding codimensions and cocharacters. Section~4 is devoted to the $(W,*)$-exponent. In Section~5, we study the $W$-polynomial identities and explicitly determine the corresponding generalized $*$-codimensions of $M^F$ and $M^D$, while Section~6 focuses on the generalized $*$-cocharacters of $M^D$. Finally, in Section~7, we investigate the growth of the varieties of $W$-algebras with involution generated by $M^F$ and $M^D$.

\section{$W$-algebras with involution} \label{s: 1}

Let $F$ be a field of characteristic zero and let $A$ be an associative algebra over $F$.
An involution on $A$ is a linear map $* \colon A \to A$ such that, for any $a,b \in A$,
$$
(a^*)^*=a, \quad \quad (ab)^*=b^*a^*.
$$
If $A$ is equipped with an involution $*$, then it is called a $*$-algebra.
Since the field $F$ has characteristic zero, a $*$-algebra $A$ can be written as $A=A^+\oplus A^-$, where
$$
A^+=\{a \in A : a^*=a\} \quad \quad \text{and} \quad \quad A^-=\{a \in A : a^*=-a\}
$$
are the vector spaces of symmetric and skew-symmetric elements of $A$, respectively.

\medskip

Now fix an $F$-algebra $W$ endowed with an involution $\sharp$. We say that a $*$-algebra $A$ is a \emph{$(W,*)$-bimodule} if it is equipped with a structure of both left and right $W$-module and the following properties hold:
\begin{itemize}
\item the corresponding actions permute, i.e., for all $w_1,w_2 \in W$, $a\in A$, we have that
\begin{equation}\label{eq: compatibility bimodule}
  (w_1a)w_2=w_1(aw_2),
\end{equation}
\item the involutions on $W$ and $A$ satisfy the following compatibility relations: for all $a \in A$, $w\in W$,
\begin{equation}\label{eq: compatibility involution}
	(wa)^*=a^*w^{\dia} \quad \text{ and } \quad (aw)^*=w^{\dia}a^*.
\end{equation}
\end{itemize}
Moreover, if $W$ has a unity element $1_W$, we further assume that $1_W a=a 1_W=a$ for all $a\in A$. 

\begin{definition}
We say that $A$ is a \emph{$(W,*)$-algebra}, or a \emph{$W$-algebra with involution $*$}, if $A$ is a $(W,*)$-bimodule satisfying the following associative relations: for all $w\in W$, $ a_1,a_2\in A$,
\begin{align}\label{eq: compat}
		w(a_1a_2)=(wa_1)a_2, \quad
		(a_1a_2)w=a_1(a_2w), \quad
		(a_1w)a_2=a_1(wa_2).
\end{align}
\end{definition}

\begin{remark}\label{rk: A op}
	Given an algebra $A$, we can define the opposite algebra $A^{\op}$ as the underlying vector space of $A$ endowed with the opposite product $\cdot^{\op}$, defined by $a	\cdot^{\op}b=b \cdot a$, where $\cdot$ is the product of $A$.
\end{remark}

Let $\End_F(A)$ be the algebra of the endomorphisms of $A$ acting on the left of $A.$ The multiplication in $\End_F(A)$ is given by the usual composition of endomorphisms, written by juxtaposition. 

\begin{definition}
Define $\M(A)$ as the subset of $ \End_F^{\op}(A) \times \End_F(A)$ of the pairs $(R,L)$ satisfying,
for all $a,b \in A$,
\begin{equation}\label{eq: comp L,R}
R(ab) = aR(b), \quad L(ab) = L(a)b, \quad R(a)b = aL(b).
\end{equation}
We refer to $R$ and $L$ as a \emph{right} and \emph{left multiplier} of $A$, respectively, and to the elements of $\M(A)$ as \emph{multipliers} of $A$.
\end{definition}

An element $m\in A $ obviously determines the multiplier $(R_m,L_m)\in \M(A)$, where 
$$
R_m(a)=am \quad \mbox{ and } \quad L_m(a)=ma
$$
for all $a\in A,$ called \emph{inner multiplier} of $A$ induced by the element $m$.

\smallskip

Our next goal is to establish a correspondence between multipliers and $(W,*)$-actions. To this end, we first record several useful properties of multipliers.

It is straightforward to verify that $\M(A)$ is a unital $F$-algebra under component-wise addition and multiplication, with unit element $(\ide_A, \ide_A).$

Now suppose that $A$ is an $F$-algebra equipped with an involution $*$.
We can endow $\M(A)$ with a corresponding involution. To this end, if $T\in \End_F(A)$, we define the map $T^*\in \End_F(A)^{op}$ as
$$
T^*(a):=(T(a^*))^*
$$
for every $a\in A$.
Note that if we take a map $S\in \End_F(A)^{op}$, one can define $S^*\in \End_F(A)$ in the same manner as before.
Thus, we equip $\M(A)$ with an involution $*$ defined by
\[
(R,L)^{*}: =(L^*,R^*).
\]
We refer to $\M(A)$ with this involution as the \emph{multiplier $*$-algebra} of $A$.
Note that this involution interchanges the right and left multipliers.

Recall that a $*$-ideal $I$ of a $*$-algebra $A$ is an ideal $I$ of $A$ such that $I^*=I$. Then, clearly, the set of all inner multipliers of $A$
$$
\IM(A)=\{(R_m,L_m) \, : \, m\in A\}
$$
is a two-sided $*$-ideal of $\M(A)$, and it is called the \emph{inner multiplier $*$-ideal} of $A$.

\begin{definition}
A right multiplier $R$ and a left multiplier $L$ of $A$ are \emph{permutable} if
\begin{equation}\label{eq: permutability}
RL = LR.
\end{equation}
\end{definition}

In general, not all left and right multipliers are permutable, as shown by the following example.

\begin{example}\label{ex: non permutable}
    Let $A=Fe_{12}\oplus Fe_{13} \oplus Fe_{23}$, where the $e_{ij}$’s are the usual matrix units. Consider the multiplier  $(R,L)\in \M(A)$ defined by
    \begin{equation*}
        \begin{split}
            R(e_{12})=e_{12}, \quad R(e_{13})=0, \quad R(e_{23})=e_{13},\\
            L(e_{12})=e_{13}, \quad L(e_{13})=0, \quad L(e_{23})=e_{23}.
        \end{split}
    \end{equation*}
Then $R$ and $L$ are not permutable. Indeed, $ R(L(e_{12}))=R(e_{13})=0,$ whereas $L(R(e_{12}))=L(e_{12})=e_{13}.$ Notice also that $(R,L)\notin \IM(A)$. Indeed, suppose there exists $m=\alpha e_{12} + \beta  e_{13} + \gamma e_{23}\in A$ such that $(R,L)=(R_m,L_m)$, then applying the right multiplier to $e_{12}$, we obtain $e_{12}=\gamma e_{13}$,  which is a contradiction.
\end{example}
 
It is straightforward to show that every multiplier permutes with any inner multiplier. Consequently, for unital algebras, we obtain the following result.

\begin{proposition}\label{pro: unitary}
    Let $A$ be a unital $*$-algebra. Then $A\cong \IM(A)= \M(A)$ as $*$-algebras and $R_aL_b=L_bR_a$ for every $a,b\in A$.
\end{proposition}
\begin{proof}
From \cite[Corollary 2.5]{MarRiz}, we have that $\IM(A)=\M(A)$, and moreover, $A\cong \IM(A)$ as $F$-algebras via the isomorphism $\mu \colon A \to \IM(A)$ defined by $\mu(m)=(R_m,L_m)$, for all $m\in A$. Since $R_{m^*}(a)=L_{m}^{*}(a)$ and $L_{m^*}(a)=R_{m}^{*}(a)$ for all $a,m\in A$, it follows that $\mu$ is an isomorphism of $*$-algebras, which completes the proof.
\end{proof}

Now we can explain the connection between multiplier $*$-algebras and $(W,*)$-algebras. Recall that the right and left actions of $W$ on $A$ induce the following $F$-algebra homomorphisms
\begin{align*}
    \rho \colon W \to \End_F(A)^{op} \quad \mbox{ and }  \quad \lambda \colon W \to \End_F(A)
\end{align*}
defined by $\rho(w)(a)=aw$ and $\lambda (w)(a)=wa$, for all $w \in W, a \in A$,
and vice versa. These are referred to as the right and left representations of $W$ on $A$, respectively.
From \eqref{eq: compat}, it follows that
$(\rho(w),\lambda(w))\in \M(A)$, for all $w \in W$.
Thus, the map
\begin{equation*}\label{eq: phi act hom}
    \begin{split}
        \Phi \colon W & \longrightarrow\M(A)\\
        w &\mapsto (\rho(w),\lambda(w))\\
    \end{split}
\end{equation*}
is a homomorphism of $F$-algebras. By \eqref{eq: compatibility involution}, for all $w\in W$, we have that
\begin{equation*}
	\rho^*(w)=\lambda(w^{\dia})
	\quad \text{and} \quad
	\lambda^*(w)=\rho(w^{\dia})
\end{equation*}
and, as a consequence, $\Phi$ is a homomorphism of $*$-algebras.
Moreover, by \eqref{eq: compatibility bimodule}, $\Phi$ satisfies the following additional condition: for all $w_1,w_2\in W$
\begin{equation}\label{Phi permutability}
    \rho(w_2)\lambda(w_1)=\lambda(w_1)\rho(w_2).
\end{equation}
Furthermore, if $W$ has unity, then $\Phi(1_W)=(\ide_A,\ide_A)$.

The converse is not generally true because not all $*$-homomorphisms from $W$ to $\M(A)$ define a $(W,*)$-algebra structure on $A$; indeed, there may exist $*$-homomorphisms from $W$ to $\M(A)$ that do not satisfy condition \eqref{Phi permutability}, since generally not all multipliers of $A$ satisfy the permutability condition \eqref{eq: permutability} (see Example \ref{ex: non permutable}). 

However, if we consider a $*$-homomorphism $\Phi \colon W \to \M(A)$ such that $\Phi(W)$ satisfies the permutability condition \eqref{eq: permutability}, then $A$ inherits a natural structure of $(W,*)$-algebra defined as follows: for any $w\in W$ and $a\in A$
$$
wa:=L(a) \quad \text{and} \quad  aw:=R(a),
$$
where $\Phi(w)=(R,L)\in \M(A)$. 

 Thus, giving a $(W,*)$-algebra structure on $A$ is equivalent to giving a $*$-homomorphism $\Phi\colon W \to \M(A)$ whose image $\Phi(W)\subseteq \M(A)$ satisfies the condition of permutability.  Such a $*$-homomorphism is called the \emph{acting $*$-homomorphism} of $W$ on $A.$ 

An algebra $A$ may admit several $(W,*)$-algebra structures corresponding to distinct acting $*$-homomorphisms $\Phi\colon W\to\M(A).$  When it is necessary to specify the $*$-homomorphism defining the action, we write $A^\Phi$ for the algebra $A$ endowed with the $(W,*)$-action determined by $\Phi$.

Notice that the action of $W$ on $A$ induces an action of the $*$-subalgebra $\overline{W}=\Phi(W)$ of $\M(A)$, so $A$ can be regarded as a $(\overline{W},*)$-algebra. Thus, the action of $W$ on $A$ factors through $\overline{W}$. 
Moreover, distinct acting $*$-homomorphisms $ \Phi\colon W\to\M(A)$ with the same image $\overline{W}$ induce the same $(\overline{W},*)$-algebra structure on $A$.  Thus we make the following definition.

\begin{definition}\label{def: equivalence}
We say that $(W,*)$-algebra structures $A^\Phi$ and
$A^{\Phi'}$ are \emph{equivalent} if $\Phi(W)=\Phi'(W)$ as $*$-subalgebras of $\M(A)$.
\end{definition}

This defines an equivalence relation on the set of $(W,*)$-algebra structures on $A$.
Its equivalence classes are precisely determined by
the $*$-subalgebras of $\M(A)$ that occur as images of acting $*$-homomorphisms.
 Accordingly, throughout the paper, we describe a $(W,*)$-action on $A$ in terms of its image $\overline{W}$, and say that
$W$ acts on $A$ via $\overline{W}$. 

Whenever it is necessary to specify the image, we write $A^{\overline{W}}$ for $A$ endowed with a fixed $(W,*)$-action whose
image is $\overline{W}$. Thus, $A^{\overline{W}}$ always denotes a fixed
representative of the corresponding equivalence class.  The acting
$*$-homomorphism $\Phi$ is consequently understood to be part of the
data, even though it is not displayed in the notation
$A^{\overline{W}}$. We will write $\Phi$ explicitly only when necessary to avoid ambiguity.

From the arguments above and Proposition \ref{pro: unitary}, the following result holds.

\begin{proposition}\label{pro: unital subalgebra}
If $A^{\overline{W}}$ is $(W,*)$-algebra with unit $1_A$, then the action of $W$ on $A$ is induced by the left and right multiplication actions of a $*$-subalgebra $B$ of $A$ that is isomorphic to $\overline{W}$. Moreover, if $W$ is unital, $1_A\in B$.
\end{proposition}

\smallskip

In the last part of this section we will deal with finite dimensional $(W,*)$-algebras. Henceforth, from now on, $A$ will denote a finite dimensional $*$-algebra and $J=J(A)$ the Jacobson radical of $A$.

Recall that $A$ is called \textit{semisimple} if $J=0$ and ($*$-)\textit{simple} if it has no-nontrivial ($*$-)ideals and $A^2\neq 0$.

We next recall the Wedderburn–Malcev theorem for algebras with involution (see \cite[Theorem 3.4.4]{GiaZai}).

\begin{theorem}\label{th: WM invol} Let $F$ be a field of characteristic 0.
	 Any finite dimensional $F$-algebra $A$ with involution $*$ has a decomposition of the form
     \begin{equation}\label{eq: W-M decomposition}
         A=B+J
     \end{equation}
	where the Jacobson radical $J=J(A)$ is a $*$-ideal and $B$ is a maximal semisimple $*$-subalgebra of $A$. If the field $F$ is algebraically closed, then
	$$
	B=B_1\oplus \cdots \oplus B_m,
	$$
    where the $B_i$s are $*$-simple subalgebras of $A$.
\end{theorem}

Given a finite dimensional $(W,*)$-algebra $A$ and a subset $I\subseteq A$, we consider
 $$
 WI=\{wa : w\in W, \text{ } a\in I\} \quad \text{ and } \quad IW=\{aw :  a\in I, \text{ } w\in W\}.
 $$
 We say that $I$ is $W$-invariant if $IW,WI\subseteq I$. A $*$-ideal ($*$-subalgebra) of $I$ of $A$ is called $(W,*)$-ideal ($(W,*)$-subalgebra) if it is $W$-invariant.

 Notice that the Wedderburn-Malcev decomposition may not be $W$-invariant. Indeed, suppose that $A$ admits a decomposition of the form \eqref{eq: W-M decomposition} with $J\neq 0$ and the acting $*$-homomorphism $\Phi$ of $W$ on $A$ satisfies $\Phi(W)\subseteq \IM(J)$, i.e., $W$ acts on $A$ via the $*$-ideal $J$ by left and right multiplications. In this case, we have $BW,\,WB\subseteq J$, and hence $B$ is not $W$-invariant. Nevertheless, the following result holds.

 \begin{theorem}\label{th: WM for W*A}
 	Let $W$ be a $\dia$-algebra over a field $F$ of characteristic zero and $A$ be a finite dimensional $(W,*)$-algebra. Then the Jacobson radical $J$ of $A$ is a $(W,*)$-ideal of $A$. Moreover, if $F$ is algebraically closed, then $A=B_1\oplus \cdots \oplus B_m+J$, where $B_i$ is a $*$-simple subalgebra such that $WB_i,\, B_iW\subseteq B_i+J$ for all $1\leq i \leq m$.
 	\end{theorem}
\begin{proof}
The result follows from Theorem \ref{th: WM invol} and by \cite[Theorem 3.3]{MarRiz}.
\end{proof}

 		We say that a $(W,*)$-algebra $A$ is $(W,*)$-simple if $A^2\neq 0$ and $A$ has no non-zero $(W,*)$-ideals. As a consequence of Theorems \ref{th: WM invol} and \ref{th: WM for W*A}, we obtain the following result, which will serve to describe the structure of $(W,*)$-simple algebras.
 		
 		\begin{corollary}\label{cor: *simple W*simple}
 			Let $A$ be a finite dimensional $(W,*)$-algebra over an algebraically closed field $F$ of characteristic zero. Then $A$ is $(W,*)$-simple if and only if it is $*$-simple.
 		\end{corollary}

 		Now, in order to describe simple $(W,*)$-algebras, we recall some notions from the theory of $*$-algebras.
 		
 		Consider the matrix algebra $A=M_n(F)$. It is well known that we can define on $A$ two involutions:
 		\begin{itemize}
 			\item the transpose involution $t$, that is the usual matrix transpose;
 			\item the symplectic involution $s$ defined, only when $n=2k$ is even, as
 			$$
 			\begin{pmatrix}
 				A & B\\
 				C & D\\
 			\end{pmatrix}^s=\begin{pmatrix}
 			D^t & -B^t\\
 			-C^t & A^t\\
 			\end{pmatrix}
 			$$
 			where $A, B, C, D\in M_{k}(F)$.
 		\end{itemize}
 		
 		From Remark \ref{rk: A op}, given an algebra $A$ we can consider the algebra $A^{op}$. Thus, the algebra $B=A \oplus A^{op}$ is a $*$-algebra endowed with the so-called exchange involution $ex$ given by
 		$$
 		(a,b)^{ex} = (b,a).
 		$$
 		
 		Finally, we can recall the characterization of $*$-simple algebras (see \cite{GiaZai,Row}).

        \begin{theorem}\label{th: *-simple characterization}
            Let $F$ be an algebraically closed field of characteristic zero. Then, up to isomorphisms, all finite dimensional $*$-simple algebras are the following ones
 		\begin{itemize}
 			\item $(M_n(F),t)$, the algebra of $n \times n$ matrices with transpose involution;
 			\item $(M_{2k}(F),s)$, the algebra of $2k \times 2k$ matrices with symplectic involution;
 			\item $(M_n(F)\oplus M_n(F)^{op},ex)$, the direct sum of the algebra of $n\times n$ matrices and its opposite algebra with the exchange involution.
 		\end{itemize}
        \end{theorem}
 		
 		From this characterization and Corollary \ref{cor: *simple W*simple} and Proposition \ref{pro: unital subalgebra}, we have the following result.
 		
 		\begin{corollary}
 			If $A$ is a finite dimensional $(W,*)$-simple algebra over an algebraically closed field $F$ of characteristic zero, then $A$ is isomorphic to one of the following $(W,*)$-algebras:
 			\begin{itemize}
 				\item $M_n(F)$ with transpose or symplectic involution,
 				\item $M_n(F)\oplus M_n(F)^{op}$ with exchange involution $(a,b)^*=(b,a)$,
 			\end{itemize}
 			with the $W$-action given by left and right multiplication of an appropriate unital $*$-subalgebra.
 		\end{corollary}

	\section{Free $(W,*)$-algebra and generalized $*$-identities}
 		
From now on, $W$ will be a unital $*$-algebra endowed with an involution $\sharp$. 

Let $A$ and $B$ be two $(W,*)$-algebras, endowed with the involutions
$\ast_A$ and $\ast_B$, respectively. A $(W,*)$-homomorphism $\phi \colon A \to B$  must satisfy the conditions $ \phi(a^{*_A})=\phi(a)^{*_B}$ and $\phi(wav)=w\phi(a)v$, for all $a\in A$, $w,v\in W$.

        We start by constructing the free $(W, *)$-algebra $W\langle X,* \rangle$. To this end, given two disjoint countable infinite sets of variables $X: =\{x_1, x_2, \ldots \}$ and $X^*: =\{x_1^*, x_2^*, \ldots \}$, we consider the free $W$-algebra $W\langle X\cup X^* \rangle$ freely generated by $X\cup X^*$ over $F$. Starting with the given involution $\dia$ on $W$, we induce an involution $*$ on $W\langle X\cup X^* \rangle$ by defining, for all $w_1,w_2 \in W$ and $x_i\in X$,
 		$$
 		(w_1x_iw_2)^*:=w_2^{\dia}x_i^*w_1^{\dia}.
 		$$
In this way $W\langle X\cup X^* \rangle$ become the free $(W,*)$-algebra over $X$ which we denote by $W\langle X, *\rangle$, and its elements are called \textit{generalized $*$-polynomials} or \textit{$(W,*)$-polynomials}.
        Thus, it is uniquely determined up to isomorphism by the following universal property:
 		given a $(W,*)$-algebra $A$, any map $\varphi\colon X \to A$ can be uniquely extended to a homomorphism of $(W,*)$-algebras $\bar{\varphi}\colon W \langle X, * \rangle \to A$, called \textit{evaluation} of $W\langle X,*\rangle$ in $A$. Notice that $\bar{\varphi}$ depends also on the acting $*$-homomorphism $\Phi$ of $W$ on $A$, so   we will denote the evaluation $\Bar{\varphi}$ by $\Bar{\varphi}_{\Phi}$, or simply $\varphi_{\Phi}$, when necessary to clarify which action we are considering.

        \smallskip

        Since  $W\langle X,* \rangle$ is, in particular, a $*$-algebra, then it decomposes as $W\langle X,* \rangle=W\langle X,* \rangle^+ \oplus W\langle X,* \rangle^-$ where the $(W,*)$-polynomials of $W\langle X,*\rangle^+$ are called symmetric, and those in $W\langle X,*\rangle^-$ are called skew-symmetric.
Since $W$ has unit $1_W$, we set $1_Wx_i1_W=x_i$ and $1_Wx_i^*1_W=x_i^*$, $i\geq 1.$ Thus, if we define
 			$$
 		x_i^+:=x_i+x_i^* \quad \text{ and } \quad x_i^-:=x_i-x_i^*, \quad i\geq 1,
 		$$
 		then $W\langle X,*\rangle$ may be regarded as the free $(W,*)$-algebra on the two sets $X^+=\{x_1^+,x_2^+,\ldots\}$ and $X^-=\{x_1^-,x_2^-,\ldots\}$, where the variables from $X^+$ are symmetric, i.e., $x_i^+\in W\langle X, * \rangle^+,$ $i\geq 1$, and the variables from $X^-$ are skew-symmetric, i.e., $x_i^-\in W\langle X, * \rangle^-$, $i\geq 1.$

Since $W$ has unity, fix a basis $\mathcal{B}_W:=\{w_i \, : \, i\geq 0\}$ of $W$, a vector-space basis of $W\langle X, * \rangle$ is
 	  $$
 	  \mathcal{B}_{W\langle X, *\rangle}= \{w_{i_0}x_{j_1}^{\epsilon_{1}}w_{i_1}x_{j_2}^{\epsilon_{2}}w_{i_2}\cdots w_{i_{n-1}}x_{j_n}^{\epsilon_{n}}w_{i_{n}} \; : \, \epsilon_k\in \{+,-\},\, w_{i_l}\in \mathcal{B}_{W},\, 1\leq k \leq n, \, 0\leq l \leq n,\, n\geq 1  \}.
 $$

 	  Given a $(W,*)$-algebra $A$, a generalized $*$-polynomial $ f=f(x_{1}^+,\ldots,x_{r}^+,x_{1}^-,\ldots ,x_{{s}}^-)\in W\langle X,*\rangle $
 	  is a \emph{$(W,*)$-identity}, or a \emph{generalized $*$-identities} of $A$, and we write $f\equiv 0$, if $f(a_{1}^+,\ldots,a_{r}^+,a_{{1}}^-,\ldots,a_{{s}}^-)=0$
 	  for all choices of $a_{1}^+,\ldots,a_{r}^+ \in A^+$, $a_{{1}}^-,\dots,a_{s}^- \in A^-$. We denote by $\I^{W,*}(A)$ the set of generalized $*$-identities of $A$, and it forms a $T_W^*$-ideal of $W\langle X,* \rangle$, that is,  a $(W,*)$-ideal invariant under all $(W,*)$-algebra endomorphisms of $W\langle X,*\rangle$. Given a set $S\subseteq W\langle X,*\rangle$, by $\langle S\rangle_{T_W^*}$ we denote the smallest $T_W^*$-ideal containing $S$.

\smallskip

 	 	For $n\geq 1$, we denote by $P_n^{W,*}$ the vector space of multilinear $(W,*)$-polynomials in the variables $x_1^+,x_1^-,\ldots,x_n^+,x_n^-$, that is
 	 	\begin{equation*}\label{eq: PnW*}
 	 		P_n^{W,*}: = \spa_F\{w_{i_0}x_{\sigma (1)}^{\epsilon_1} w_{i_1}x_{\sigma(2)}^{\epsilon_2}w_{i_2}\cdots w_{i_{n-1}}x_{\sigma(n)}^{\epsilon_n}w_{i_{n}}\, : \, \epsilon_k\in \{+,-\},\, w_{i_l}\in \mathcal{B}_{W},\, 1\leq k \leq n,\, 0\leq l \leq n,\, \sigma \in S_{n} \}.
 	 	\end{equation*}

\begin{remark}
    Since $W$ is a $*$-algebra, we can consider a basis $\mathcal{B}_W$ of $W$ which is the union $\mathcal{B}_W=\mathcal{B}_{W^+}\cup \mathcal{B}_{W^-}$ of bases of $W^+$ and $W^-$, respectively, where $\mathcal{B}_{W^+}:=\{w_i^+ \, : \, i\geq 1\}$ and $\mathcal{B}_{W^-}:=\{w_i^- \, : \, i\geq 1\}$. Then, for $n \geq 1$ we can write
    \begin{align*}
        P_n^{W,*} = \spa_F\{w_{i_0}^{\eta_0}x_{\sigma (1)}^{\epsilon_1} w_{i_1}^{\eta_1}x_{\sigma(2)}^{\epsilon_2}w_{i_2}^{\eta_2}\cdots w_{i_{n-1}}^{\eta_{n-1}}x_{\sigma(n)}^{\epsilon_n}w_{i_{n}}^{\eta_n}\, : \, \epsilon_k,\eta_l\in \{+,-\},\, w_{i_l}^{\eta_l}\in \mathcal{B}_{W^{\eta_l}},  \ \ \ &\\
         1\leq k \leq n,\, 0\leq l \leq n,\, \sigma \in S_{n}  &\ \}.
    \end{align*}
\end{remark}

 	 Since $\mbox{char}F = 0$, a standard linearization argument shows that the $T_W^*$-ideal $\I^{W,*}(A)$ is completely determined by its multilinear generalized $*$-polynomials. Thus, we can consider the quotient space
 	  $$
 	  P_n^{W,*}(A): =\frac{P_n^{W,*}}{P_n^{W,*}\cap \I^{W,*}(A)},
 	  $$
 	  and the \textit{$n$-th $(W,*)$-codimension}, or \textit{generalized $*$-codimension}, of $A$ is defined as
 	  $$
 	  c_n^{W,*}(A): = \dim_F P_n^{W,*}(A).
 	  $$

 	   Note that $c_n^{W,*}(A)$ is not necessarily finite. However, if at least one between $W$ and $A$ is finite dimensional, then $c_n^{W,*}(A)$ is finite for all $n \geq 1$.

       Remark that when $W\cong F$ we are dealing with the ordinary $*$-identities. In this case we write $P_n^*:=P_n^{F,*}$, $\I^*(A):=\I^{F,*}(A)$ and $c_n^*(A):=c_n^{F,*}(A)$. More generally, since $W$ has unity, $F$ embeds naturally as a subalgebra of $W$ and it follows that $P_n^* \subseteq P_n^{W,*}$. Consequently, for any $(W,*)$-algebra $A$, the ordinary $*$-codimensions are bounded above by the generalized $*$-codimensions. That is,
    \begin{equation} \label{eq: lower}
        c_n^*(A) \leq c_n^{W,*}(A), \quad n \geq 1.
    \end{equation}

The study of multilinear generalized $*$-polynomial can be reduced to the one of smaller spaces in the following way. For $n\geq 1$ and $0\leq r \leq n$, define $P_{r,n-r}$ as the subspace of $P_n^{W,*}$ of multilinear $(W,*)$-polynomials in the variable $x_1^+, \ldots, x_r^+, x_{r+1}^-,\ldots, x_n^-$, i.e., the first $r$ variables are symmetric and the last $n-r$ are skew-symmetric. It is clear that $P_n^{W,*}$ is isomorphic to the direct sum of $\binom{n}{r}$ copies of $P_{r,n-r}$ with $0\leq r \leq n$ as well as that $P_n^{W,*}(A)$ inherits this decomposition. Thus, one sets
$$
P_{r,n-r}(A):=\frac{P_{r,n-r}}{P_{r,n-r}\cap \I^{W,*}(A)}
$$
and
$$
c_{r,n-r}(A):=\dim_F P_{r,n-r}(A),
$$
and as a consequence, for all $n\geq 1$ we have that
\begin{equation}\label{eq: decomp P_n^{W,*}(A)}
    P_n^{W,*}(A)\cong \bigoplus_{0\leq r \leq n}\binom{n}{r} P_{r,n-r}(A)
\end{equation}
and
\begin{equation}\label{eq: decomp codim}
    c_n^{W,*}(A)= \sum_{r=0}^n\binom{n}{r}c_{r,n-r}(A).
\end{equation}

We can attach another useful invariant to the spaces $P_{r,n-r}(A)$ in the following way. For $n \geq 1$ and $0\leq r\leq n$
we consider the left permutation action of the group $S_{r} \times S_{n-r}$ on the space  $P_{r,n-r}$ by letting $S_r$ act on $x_1^+, \ldots, x_r^+$ and $S_{n-r}\equiv S_{\{r+1, \ldots, n\}}$ on $ x_{r+1}^-,\ldots, x_n^-$. Since $\I^{W,*}(A)$ is invariant under this $S_r \times S_{n-r}$ action, we have that $P_{r,n-r}(A)$ is a $S_{r}\times S_{n-r}$ left module. The corresponding character, denoted by $\chi_{r,n-r}^{W,*}(A)$ is called \emph{$(r, n-r)$-th $(W,*)$-cocharacter}, or \emph{generalized $*$-cocharacter}, of $A$. 

Recall that the irreducible $(S_r \times S_{n-r})$-characters are the tensor products $\chi_\lambda \otimes \chi_\mu$ of the irreducible $S_r$-character $\chi_\lambda$ and $S_{n-r}$-character $\chi_\mu$, where $\lambda\vdash r$ and $\mu \vdash n-r$ are partitions.
Since $F$ is of characteristic zero, by complete reducibility we can write
\begin{equation}\label{eq: ch}
    \chi_{r,n-r}^{W,*}(A)=\sum_{(\lambda,\mu)\vdash (r,n-r)}m_{\lambda,\mu}(\chi_\lambda \otimes \chi_\mu)
\end{equation}
where $\lambda\vdash r$, $\mu \vdash n-r$, and $m_{\lambda, \mu}\geq 0$ is the multiplicity corresponding to $\chi_\lambda \otimes \chi_\mu$.

\smallskip

 	Since there is a one-to-one correspondence between $T_W^*$-ideals and varieties of $(W,*)$-algebras, it is often convenient to translate all the objects we have introduced to the language of varieties.
    A variety $\mathcal{V}$ of $(W,*)$-algebras generated by a $(W,*)$-algebra $A$ is denoted by $\V^{W,*}(A)$ and is called \textit{$(W,*)$-variety}, or \textit{generalized $*$-variety}. We define $\I^{W,*}(\mathcal{V}) : = \I^{W,*}(A)$. Moreover, the growth of $\mathcal{V}$ is the growth of the sequence $c_n^{W,*}(\mathcal{V}) : = c_n^{W,*}(A)$, $n \leq 1$, and $\chi_{r,n-r}^{W,*}(\mathcal{V}):= \chi_{r,n-r}^{W,*}(A)$, for all $n \geq 1$ and $0\leq r \leq n$, is the $(r,n-r)$th $(W,*)$-cocharacter of $\mathcal{V}$.

 \medskip

 We conclude this section with a result on the $(W,*)$-variety generated by finite-dimensional unital $(W,*)$-algebras, which will be useful in what follows.

    Let $A^{\overline{W}}$ be a finite-dimensional $(W,*)$-algebra and let $\Phi\colon W\to \M(A)$ be the corresponding acting $*$-homomorphism of $W$ on $A$, i.e., such that $\overline{W}=\Phi(W)$.
    Since $\overline{W}$ is also finite dimensional and $F$ is of characteristic $0$, the Wedderburn–Malcev decomposition for algebras with involution (Theorem \ref{th: WM invol}) yields
 $$
    \overline{W}=\overline{W}_{ss}+\bar{J}
 $$
 where $\overline{W}_{ss}$ is a semisimple $*$-subalgebra of $\overline{W}$ isomorphic to $\overline{W}/\bar{J}$ and $\bar{J}=J(\overline{W})$ is the Jacobson radical of $\overline{W}$. If $\pi \colon \overline{W}\rightarrow \overline{W}/\bar{J}$ denotes the canonical epimorphism, then, via the above identification, we may assume that $\pi \colon \overline{W} \rightarrow \overline{W}_{ss}$.

 Define a map $\Phi_{ss} \colon W \to \M(A)$ by $\Phi_{ss}  := \pi \Phi.$ Then $\Phi_{ss}$ is a homomorphism of $*$-algebras satisfying $\Phi_{ss}(W)=\overline{W}_{ss}$ and therefore determines another $(W,*)$-algebra structure on $A$. Therefore, $A^{\overline{W}_{ss}}$ is a $(W,*)$-algebra defined by the acting $*$-homomorphism $\Phi_{ss}$. 

 We can now state the following result, whose proof follows from that of \cite[Theorem 4.3]{MarRiz} with mild and natural adaptations.
 	
 \begin{theorem}\label{th: ss}
 Let $W$ be a unital algebra with involution over a field $F$ of characteristic zero, and $A^{\overline{W}}$ be a finite-dimensional unital
    $(W,*)$-algebra.
     If $\Bar{J}\subseteq \IM(J(A))$, where $\Bar{J}=J(\overline{W})$ and $J(A)$ are the Jacobson radicals of $\overline{W}$ and $A$, respectively, then $A^{\overline{W}_{ss}}\in \V^{W,*}(A^{\overline{W}}).$
 \end{theorem}

 	   \section{The $(W,*)$-exponent}

In \cite{GiaPolVal} the authors proved that, for any $*$-algebra $A$ that satisfies a non-trivial identity the limit
$$
    \exp^*(A) := \lim_{n \to \infty} \sqrt[n]{c_n^*(A)}
$$
exists and is an integer, referred to as the $*$-exponent of $A$.
Moreover, in \cite{Gordienko2013} the same result was established for a broad class of algebras, which includes algebras with involution, under the additional assumption of finite dimensionality.
By analogy, we define the generalized $(W,*)$-exponent of $A$ by
$$
    \exp^{W,*}(A) := \lim_{n \to \infty} \sqrt[n]{c_n^{W,*}(A)}.
$$
The main objective of this section is to show that, for any finite-dimensional $(W,*)$-algebra $A$, this limit exists and coincides with $\exp^*(A)$.

From this point on, we therefore assume that $A$ is a finite-dimensional $(W,*)$-algebra. We begin by recalling the following definition.

 	\begin{definition}
    Let $A$ be a finite-dimensional $(W,*)$-algebra such that $A=B_1 \oplus \cdots \oplus B_m + J$, where $B_1,\ldots, B_m$ are $*$-simple algebras.
 	 We say that a $*$-subalgebra $B_{i_1}\oplus \cdots \oplus B_{i_k}$, where the $B_{i_j}$'s are distinct simple $*$-algebras, is \textit{admissible} if for some permutation $(l_1,\ldots,l_k)$ of $(i_1,\ldots, i_k)$ we have that
 	 $$B_{l_1}JB_{l_2}J\cdots JB_{l_k}\neq 0.$$
 	 \end{definition}

We start by proving the lower bound.

\begin{lemma}\label{lem: lower bound}
    Let $A$ be a finite dimensional $(W,*)$-algebra over an algebraically closed field $F$ of characteristic zero. Then there exist constants $C>0$ and $t$ such that
    $$
    c_n^{W,*}(A)\geq Cn^t d^n,
    $$
    where $d$ is the maximal dimension of an admissible $*$-subalgebra of $A$.
\end{lemma}
\begin{proof}
    By \cite[Theorem 5]{Gordienko2013} and \eqref{eq: lower} there exist constants $C>0$ and $t$ such that $c_n^{W,*}(A)\geq c_n^* (A) \geq Cn^t d^n$,
    where $d$ is the maximal dimension of an admissible $*$-subalgebra of $A$, as required.
\end{proof}

In order to compute the upper bound we need the following construction.
 Let $A = A^+ \oplus A^-$ be a finite dimensional $(W,*)$-algebra.
 	We fix a basis $\mathcal{B}_A$ of $A$ which is the union $\mathcal{B}_A = \mathcal{B}^+ \cup \mathcal{B}^-$ of bases of $A^+$ and $A^-$, respectively, where $\mathcal{B}^+ = \{ a_1, \ldots, a_s \}$ and $\mathcal{B}^- = \{ b_1, \ldots, b_t \}$. Moreover, for a fixed $n \geq 1$ consider $n(s+t)$ commutative variables $\xi_{i,j}$, $1 \leq i \leq n$, $1\leq j\leq s+t$. Then, for $1\leq i \leq  n$, we define the generic elements:
 	\begin{equation}\label{eq: generic elem}
 			\xi_i^+ := \sum_{j=1}^s
 			\xi_{i,j}\otimes a_{j}
 			\quad \text{and} \quad
 			\xi_i^- := \sum_{j=1}^t
 			\xi_{i,s+j}\otimes b_{j}.
 	\end{equation}
    We denote by $H$ the algebra generated by the $\xi_i^+$ and $\xi_i^-$, $1\leq i \leq n$, over $F$. Then, the $(W,*)$-algebra structure of $A$ induces naturally a corresponding $(W,*)$-algebra structure on $H$. The involution $*$ on $H$ is defined by assuming that the elements $\xi_1^+, \ldots, \xi_n^+$ are symmetric and $\xi_1^-, \ldots, \xi_n^-$ are skew-symmetric.
    The action of $W$ on $H$ is given as follows: for $w,v \in W$ and $1\leq i \leq n$,
 	\begin{equation*}\label{eq: gen variabili}
 			w\xi_i^+v= \sum_{j=1}^s
 			\xi_{i,j}\otimes (wa_{j}v)
 			\quad \text{and} \quad
 			w\xi_i^-v= \sum_{j=1}^t
 			\xi_{i,s+j}\otimes (wb_{j}v).
 	\end{equation*}
Moreover, let $F[\xi_{i,j}]:=F[\xi_{i,j} | 1\leq i \leq n, 1\leq j \leq s+t]$ denote the polynomial ring over $F$ in the indeterminates $\xi_{i,j}$, that is, the free commutative algebra over $F$ in $n(s+t)$ variables. Then also the tensor product algebra $F[\xi_{i,j}] \otimes A$ inherits a natural $(W,*)$-algebra structure induced by the one of $A$, and we have that $H \subseteq F[\xi_{i,j}] \otimes A$.

\begin{proposition}\label{pro: relatively free algebra}
Let $W\langle x_1^+, \ldots, x_n^+,x_1^-,\ldots , x_n^-\rangle $ be the free $(W,*)$-algebra in $n$ symmetric variables and $n$ skew-symmetric variables, and  $$\psi\colon W\langle x_1^+, \ldots, x_n^+,x_1^-,\ldots , x_n^-\rangle \to H$$
be the $(W,*)$-homomorphism induced by $\psi(x_i^+)=\xi_i^+$ and $\psi(x_i^-)=\xi_i^-$, $1\leq i \leq n$. Then 
$$
\ker\psi=\I^{W,*}(A)\cap W\langle x_1^+, \ldots, x_n^+,x_1^-,\ldots , x_n^-\rangle.
$$

\end{proposition}
\begin{proof}
    We start by proving that $\I^{W,*}(A)\cap W\langle x_1^+, \ldots, x_n^+,x_1^-,\ldots , x_n^-\rangle \subseteq \ker\psi.$ To this end, let $f=f(x_1^+, \ldots, x_n^+,x_1^-,\ldots , x_n^-) \in
     W\langle x_1^+, \ldots, x_n^+,x_1^-,\ldots , x_n^-\rangle$ be a $(W,*)$-identity of $A$.
    Since the field $F$ is of characteristic zero, we can assume that $f$ is a multilinear $(W,*)$-polynomial. Then, by \eqref{eq: generic elem}, $ \psi(f)=f(\xi_1^+, \ldots, \xi_n^+,\xi_1^-,\ldots , \xi_n^-)$ is a sum of the expressions of the form
    $$\bar{f}=f(\xi_{1,j_1} \otimes a_{j_1}, \ldots , \xi_{n,j_n} \otimes a_{j_n}, \xi_{1,l_1} \otimes b_{l_1}, \ldots , \xi_{n,l_n} \otimes b_{l_n})= \xi_{1,p_1} \cdots \xi_{n,p_n} \otimes f(  a_{j_1}, \ldots ,  a_{j_n},  b_{l_1}, \ldots ,  b_{l_n}),$$
    where $p_i\in \{j_i,l_i\}$ for $1\leq i \le n$. Since $f\in \I^{W,*}(A)$, we get that $\bar{f}=0$, and as a consequence $f\in \ker \psi$, as required.

    Suppose now that $f(x_1^+, \ldots, x_n^+,x_1^-,\ldots , x_n^-)\in \ker \psi$, i.e., $f(\xi_{1}^+, \ldots,\xi_{n}^+, \xi_{1}^-,\ldots, \xi_{n}^-)=0 $ in $H$, and let $c_1,\ldots, c_n$ be arbitrary symmetric elements of  $A$ and $d_1,\ldots, d_n$ be arbitrary skew-symmetric elements of  $A$. Write each $c_i$ as a linear combination of the basis $\mathcal{B}^+$ of $A^+$ and $d_i$ as a linear combination of the basis $\mathcal{B}^-$ of $A^-$. Then, for $1\leq i \leq n$, there exist $\alpha_{i,1}, \ldots, \alpha_{i,s}, \beta_{i,s+1},\ldots,\beta_{i,s+t}\in F$ such that
    $$
    c_i=\sum_{j=1}^s \alpha_{i,j}a_{j} \quad \text{and} \quad d_i=\sum_{j=1}^t \beta_{i,s+j}b_{j}.
    $$
    By the universal property of the free commutative algebra $F[\xi_{i,j}]$ and of the tensor product, we can consider the $(W,*)$-homomorphism $\rho: F[\xi_{i,j}] \otimes A \to A$ which restricts to the identity on $A$ and, for each $1\leq i \leq n$, satisfies
$$\rho(\xi_{i,j})=\begin{cases}
    \alpha_{i,j}, & \text{if }1\leq j \leq s\\
    \beta_{i,j}, & \text{if }s+1\leq j \leq s+t
\end{cases}.$$
Since $\rho(\xi_i^+)=c_i$ and $\rho(\xi_i^-)=d_i$, $1\leq i \leq n$, we obtain
$$0=\rho(f(\xi_{1}^+, \ldots,\xi_{n}^+, \xi_{1}^-,\ldots, \xi_{n}^-))=f(\rho(\xi_{1}^+)\ldots,\rho(\xi_{n}^+), \rho(\xi_{1}^-),\ldots, \rho(\xi_{n}^-))=f(c_1,\ldots,c_n,d_1,\ldots, d_n).$$
Since $c_1,\ldots,c_n,d_1,\ldots,d_n$ are arbitrary elements of $A$, $f(x_1^+, \ldots, x_n^+,x_1^-,\ldots , x_n^-)$ is a $(W,*)$-identity of $A$. Consequently, $\ker\psi=\I^{W,*}(A)\cap W\langle x_1^+, \ldots, x_n^+,x_1^-,\ldots , x_n^-\rangle$.
\end{proof}
 	
 	As a consequence we have that $H$ is the relatively free $(W,*)$-algebra in $n$ symmetric generators and $n$ skew-symmetric generators of the $(W,*)$-variety generated by $A$, and that a $(W,*)$-polynomial is a $(W,*)$-identity if and only if it vanishes on the generic elements $\xi_{i,j}$, $1\leq i \leq n$, $1\leq j \leq s+t$. We use this in the next lemma to determine the upper bound.

 	\begin{lemma}\label{lem: upper bound}
 		Let $A$ be a finite dimensional $(W,*)$-algebra over an algebraically
 		closed field $F$, where $W$ is a unital $F$-algebra. Then there exist constants $C$, $u$ such that
 		$$
 		c_n^{W,*}(A)\leq Cn^ud^n,
 		$$
 		where $d$ is the maximal dimension of a admissible $*$-subalgebra of $A$.
 	\end{lemma}
    \begin{proof}
        Consider the $(W,*)$-homomorphism $\psi\colon W\langle x_1^+, \ldots, x_n^+,x_1^-,\ldots , x_n^-\rangle \to H$ of the proposition above. Then $\psi(P_n^{W,*})\cong P_n^{W,*}(A)$, and therefore, to estimate $c_n^{W,*}(A)$, it suffices to obtain an upper bound for $\dim_F \psi(P_n^{W,*}).$ To this end, for each $0\leq r \leq n$, set $\mathcal{P}_{r,n-r}:=\psi (P_{r,n-r})$, that is, the vector space of multilinear $(W,*)$-polynomials in the variables $\xi_1^+, \ldots, \xi_r^+, \xi_{r+1}^-, \ldots, \xi_n^-$. We begin by determining an upper bound for $\dim_F \mathcal{P}_{r,n-r}$.

       Consider a monomial $w_{i_0}\xi_{\sigma(1)}^{\epsilon_{\sigma(1)}} w_{i_1} \cdots w_{i_{n-1}} \xi_{\sigma(n)}^{\epsilon_{\sigma(n)}}w_{i_n}\in \mathcal{P}_{r,n-r} $, where $\sigma\in S_n$, $w_{i_l}\in \mathcal{B}_W$ for $1\leq l \leq n$, $\epsilon_i=+$ for $1\leq i \leq r$, and $\epsilon_i=-$ for $r+1\leq i \leq n$.
       Using \eqref{eq: generic elem}, such a monomial can be expressed as a linear combination of products of $n$ elements from the fixed basis $\mathcal{B}_A=\mathcal{B}^+\cup\mathcal{B}^-$ of $A$ with coefficients that are polynomials in the variables $\xi_{ij}$. More precisely, when expanded with respect to the basis $\mathcal{B}_A$, it is a linear combination of elements of the form $c\otimes \xi_{1,j_1} \cdots \xi_{n,j_{n}}$, where $c\in \mathcal{B}_A.$ Thus, we will estimate $\dim_F \mathcal{P}_{r,n-r}$ through an estimate of the number of possible monomials $\xi_{1,j_1} \cdots \xi_{n,j_{n}}$ of this last type which can appear as coefficients of $w_{i_0}\xi_{\sigma(1)}^{\epsilon_{\sigma(1)}} w_{i_1} \cdots w_{i_{n-1}} \xi_{\sigma(n)}^{\epsilon_{\sigma(n)}}w_{i_n}$.

       Since $F$ is algebraically closed of characteristic zero, by Theorem \ref{th: WM for W*A} we write $A=B_1\oplus \cdots \oplus B_m+J$, where $J$ is the Jacobson radical of $A$ and each $B_i$ is a simple $*$-algebra satisfying $WB_l,B_lW \subseteq B_l+J$ for $1\leq l \leq m$. We may further assume that the chosen basis $\mathcal{B}_A=\mathcal{B}^+\cup\mathcal{B}^-$ of $A$ is adapted to this decomposition, in the sense that both $\mathcal{B}^+$ and $\mathcal{B}^-$ consist of elements drawn from $J$ and of each of the $B_i$.
       Since each commutative variable $\xi_{i,j}$ in \eqref{eq: generic elem} is attached to an element of either $\mathcal{B}^+$ or $\mathcal{B}^-$, by a slight abuse of terminology, we therefore refer to $\xi_{i,j}$ as a symmetric radical, skew-symmetric radical, symmetric semisimple, or skew-symmetric semisimple variable.

       Take a monomial $ \xi_{1,j_1} \cdots \xi_{n,j_{n}}$ of degree $n$ that contains $i_1$ symmetric radical variables and $i_2$ skew-symmetric radical variables. Since $J$ is a nilpotent $(W,*)$-ideal of $A$, such a monomial is nonzero only if $i_1+i_2\leq u$, where $u$ is such that $J^u\neq 0$ and $J^{u+1}= 0$. 
       Now each such monomial involves $r-i_1$ symmetric semisimple variables and $n-r-i_2$ skew-symmetric semisimple variables which come from some distinct $*$-simple components $B_{l_1}, \ldots, B_{l_p}.$ Since $W B_l, B_l W\subseteq B_l+J$, there exist distinct indices $q_1, \ldots, q_v \in \{l_1,\ldots, l_p\}$, with $1\leq v\leq p$ such that $D:=B_{q_1}\oplus \cdots \oplus B_{q_v}$ is an admissible $*$-subalgebra of $A$.

       Now, fix an admissible $*$-subalgebra $D$ of $A$. 
 	 	We recall that each monomial has $r$ coefficients taken in the set $\{\xi_{ij} \, : \, 1 \leq i \leq r, 1\leq j \leq t\}$ and $n-r$ taken in the set $\{\xi_{ij} \, : \, r+1 \leq i \leq n, 1 \leq j \leq s\}$ and also the number of possible distributions of $i_1\leq r$ symmetric radical variables and $i_2\leq n-r$ skew-symmetric radical variables in a non-zero monomial is at most $\binom{r}{i_1}\binom{n-r}{i_2}(\dim_F J)^{i_1+i_2}\leq C_1 n^u$, where $C_1$ is a constant. So, the number of possible monomials with symmetric and skew-symmetric semisimple variables coming from $D$ is bounded by $C_2n^u(\dim D^+)^r(\dim D^-)^{n-r}.$

        Now, let $M$ be the number of maximal admissible $*$-subalgebras of $A$ and let $E$ be an admissible $*$-subalgebra of maximum dimension. Then, an upper bound for the number of possible nonzero monomials is
 	 	$MC_2n^u(\dim E^+)^r(\dim E^-)^{n-r}.$
 	 	Since we rewrote any product of $n$ basis elements of $A$ as a linear combination of basis elements, we get that
$$
 	 	\dim_F \mathcal{P}_{r,n-r}\leq C_3n^u(\dim_F E^+)^r(\dim_F E^-)^{n-r},
 	 	$$
 	 	with $C_3$ constant. Since, by Proposition \ref{pro: relatively free algebra}, $\mathcal{P}_{r,n-r}\cong P_{r,n-r}(A)$, we have that $\dim_F \mathcal{P}_{r,n-r}=c_{r,n-r}(A).$ Thus, if $d=\dim_F E^+ + \dim_F E^-=\dim_F E,$ by \eqref{eq: decomp codim} we obtain
    $$
    c_n^{W,*}(A)= \sum_{r=0}^n\binom{n}{r}\dim_F \mathcal{P}_{r,n-r} \leq  C_3n^u \sum_{r=0}^n\binom{n}{r}(\dim_F E^+)^r(\dim_F E^-)^{n-r}=C_3n^ud^n,
 	$$ 	
    as required.
    \end{proof}

    Finally, we can prove the main result of the section.

    \begin{theorem}
        \label{th: w*exp}
 	Let $A$ be a finite dimensional $(W,*)$-algebra over a field $F$ of characteristic zero, where $W$ is a unital algebra with involution. Then the $(W,*)$-exponent $\exp^{W,*}(A)$ of $A$ exists and is equal to $*$-exponent $\exp^*(A)$ of $A$.

    \end{theorem}
    \begin{proof}
         As in the ordinary case, a careful adaptation of the proof of \cite[Theorem 4.1.9]{GiaZai} shows that the codimensions remain unchanged under extension of the base field $F$, since $F$ is infinite.
             Then by Lemmas~\ref{lem: lower bound} and~\ref{lem: upper bound} there exist constants $C_1,C_2,t_1,t_2$ such that $C_1 \geq 0$ and
$$
C_1n^{t_1}d^n \leq c_n^{W,*}(A) \leq C_2n^{t_2}d^n,
$$
where $d$ is the maximal dimension of an admissible $*$-subalgebra of $A$. Thus, the $(W,*)$-exponent $\exp^{W,*}(A)$ of $A$ exists and is equal to $d$. Moreover, by \cite[Theorem~5]{Gordienko2013}, we also have $\exp^*(A)=d$, completing the proof.
    \end{proof}

As a consequence of the previous theorem, we obtain that $\exp^{W,*}(A)=d$, where $d$ denotes the maximal dimension of an admissible $*$-subalgebra of $A$. In particular, this yields the integrality of the $(W,*)$-exponent. We therefore obtain the following corollary.
\begin{corollary}
    If $A$ is a finite dimensional $(W,*)$-algebra, then the sequence of $(W,*)$-codimensions $c^{W,*}_n(A)$, $n\geq 1,$ either has exponential growth or is polynomially bounded.
\end{corollary}

We say that the generalized $*$-variety $\mathcal{V}$ has \emph{polynomial growth} if $c_n^{W,*}(\mathcal{V})$ is polynomially bounded and $\mathcal{V}$ has \emph{almost polynomial growth} if $c_n^{W,*}(\mathcal{V})$ is not polynomially bounded but every proper $(W,*)$-subvariety $\mathcal{U}$ of $\mathcal{V}$ has polynomial growth. Our interest lies precisely in the study of this boundary case.

 \section{Generalized $*$-identities of $M$}

 	 In this section we introduce the $(W,*)$-algebra $M$, which is a subalgebra of the $4 \times 4$ upper triangular matrices over $F$, and we determine its generalized $*$-identities.

Let $e_{ij}$ denote the usual matrix units. Then $M$ is the $*$-algebra
 	 $$ M=F(e_{11}+e_{44}) + F(e_{22}+e_{33}) + Fe_{12} + Fe_{34},
 	 $$
 	endowed with the reflection involution $*$, i.e., the involution given by the formula: for
$\alpha, \beta, \gamma, \delta\in F$
$$
\big(\alpha(e_{11}+e_{44})+\beta(e_{22}+e_{33})+\gamma e_{12}+ \delta e_{34}\big)^*= \alpha(e_{11}+e_{44})+\beta(e_{22}+e_{33})+\delta e_{12}+ \gamma e_{34}.
$$
 	As the field $F$ has characteristic zero, $M$ decomposes as the direct sum of its symmetric and skew-symmetric parts, $M=M^+\oplus M^-$, where
 	$$
M^+=F(e_{11}+e_{44}) + F(e_{22}+e_{33}) + F(e_{12}+e_{34}) \quad  \mbox{ and }  \quad
 	M^-= F(e_{12}-e_{34}).
 	$$

 	Our first goal in this section is to classify, up to equivalence, all possible $(W,*)$-algebra structures on $M$. To this end, let $\Phi\colon W\to \M(M)$ be an acting $*$-homomorphism of $W$ on $M$ and set $\overline{W}=\Phi(W)$.
    Since $M$ is unital, Proposition \ref{pro: unital subalgebra} implies that the action of $W$ on $M$ is induced by the left and right multiplication actions of a $*$-subalgebra $B$ of $M$ isomorphic to $\overline{W}$ and containing the unit $1_M=e_{11}+e_{22}+e_{33}+e_{44}$.
    
    Notice that,  up to isomorphisms, the proper $*$-subalgebras of $M$ containing $1_M$ are
    $$
    F1_M\cong F, \quad \quad D:= F(e_{11}+e_{44})\oplus F(e_{22}+e_{33}),
    $$
    $$
    C_1:= F1_M + Fe_{12}+ Fe_{34}, \quad C_2:= F1_M + F(e_{12}+ e_{34}), \quad  C_3:= F1_M + F(e_{12} - e_{34}),
    $$
    endowed with the corresponding induced involution.
    
    Consequently, by Definition \ref{def: equivalence} and Proposition \ref{pro: unital subalgebra}, $M$ admits exactly six pairwise non-equivalent $(W,*)$-algebra structures according to which $*$-subalgebra of $M$ induce the action by left and right multiplication. We denote these structures by 
    $$
    M^F,\qquad M^D,\qquad M^{C_i}\ (1\leq i\leq 3),\qquad M,
    $$
    according as $\overline{W}$ is isomorphic to $F$, $D$, $C_i$ for some $1\leq i\leq 3$, or to the full algebra $M$, respectively.

Notice that as a direct consequence of Theorem \ref{th: w*exp}, we have that
$$\exp^{W,*}(M^F)=\exp^{W,*}(M^D)=\exp^{W,*}(M^{C_i})=\exp^{W,*}(M)=2,$$
where $1\leq i\leq 3.$ Moreover, by Theorems \ref{th: WM for W*A} and \ref{th: ss} it follows that
$$
M^F\in \V^{W,*}(M^{C_i}), \, 1\leq i\leq 3, \qquad \text{and} \qquad M^D\in \V^{W,*}(M).
$$
Accordingly, since we aim to study the boundary case between polynomial and exponential growth, we proceed to determine generators for the $T_W^*$-ideals and to compute the exact values of the corresponding $(W,*)$-codimensions of $M^F$ and $M^D$. We begin with the following technical lemma.

    \begin{lemma}\label{Lem: consequence x_1^- x_2^-}
        For any $k\geq 0$ we have $x_{k+1}^- x_1^+ \cdots x_k^+ x_{k+2}^-\in \langle x_1^- x_2^-\rangle_{T_W^*}$.
    \end{lemma}
    \begin{proof}
    If $k=0$, there is nothing to prove, so assume $k\geq 1$. We proceed by induction on $k$.

    First observe that $x_1^+ x_2^- + x_2^- x_1^+\in W\langle X, *\rangle^-$.
    Then for the base case $k=1$,
        substituting $x_1^-$ with  $x_1^+ x_2^- + x_2^- x_1^+$ in the $(W,*)$-polynomial $x_1^-x_2^-$, and then renaming the variables, yields $x_1^+ x_2^-x_3^- + x_2^- x_1^+x_3^-\in \langle x_1^- x_2^-\rangle_{T_W^*}$. Hence $x_2^- x_1^+x_3^-\in \langle x_1^- x_2^-\rangle_{T_W^*}$.

        Now assume inductively that $x_{k}^- x_1^+ \cdots x_{k-1}^+ x_{k+1}^-\in \langle x_1^- x_2^-\rangle_{T_W^*}$. Replacing $x_k^-$ with  $x_k^+ x_{k+2}^- + x_{k+2}^- x_k^+$ in this $(W,*)$-polynomial and renaming variables, we obtain
        $x_{k+1}^- x_1^+ \cdots x_{k-1}^+ x_{k+2}^- x_k^+ + x_{k+1}^- x_1^+ \cdots x_k^+ x_{k+2}^-\in \langle x_1^- x_2^-\rangle_{T_W^*}$. Consequently, by the inductive hypothesis,  $x_{k+1}^- x_1^+ \cdots x_k^+ x_{k+2}^-\in \langle x_1^- x_2^-\rangle_{T_W^*}$.
    \end{proof}

  Let $[x_1,x_2]:=x_1x_2-x_2x_1$ denote the commutator of $x_1$ and $x_2$. More generally, for $k\geq 2$, we write $[x_1,x_2,\ldots, x_k]$ for the left-normed commutator.

  With this notation, we are ready to prove the main results of this section.

		\begin{theorem}\label{thm: id e cn M^D}
        Let $M^D$ be a $(W,*)$-algebra $M$ where $W$ acts on it as the algebra $D$ by left and right multiplication.
   Then there exists a basis $\mathcal{B}_W=\mathcal{B}_{W^+}\cup \mathcal{B}_{W^-}$  of $W$, where $\mathcal{B}_{W^+}:=\{1_W, \, w_i^+\, : \, i\geq 1\}$ and  $\mathcal{B}_{W^-}:=\{w_i^-\, : \, i\geq 1\}$ are basis of $W^+$ and $W^-$, respectively, such that  
 the $T_W^*$-ideal $\I^{W,*}(M^D)$  of generalized $*$-identities of $M^D$ is generated, as  $T_W^*$-ideal, by the following $(W,*)$-polynomials:
        $$
        x_1^-x_2^-, 	\quad w_1^+x^-+x^-w_1^+, \quad (w_1^+)^2 x^\epsilon - x^\epsilon,
            \quad  w_i^+x^\epsilon, \quad w_j^-x^\epsilon,
        $$
        where $\epsilon\in \{+,-\}$, $i\geq 2$ and $j\geq 1$.
        Moreover  $c_n^{W,*}(M^D)=n2^{n+1}+2$. 
 	\end{theorem}
 	\begin{proof}
Let $\Phi$ denote the acting homomorphism of $W$ on $M^D$. Then $\Phi(W)\cong D$ and we can choose a basis $\mathcal{B}_W=\mathcal{B}_{W^+}\cup \mathcal{B}_{W^-}$  of $W=W^+\oplus W^-$, where $\mathcal{B}_{W^+}:=\{ w_i^+\, : \, i\geq 0\}$ and  $\mathcal{B}_{W^-}:=\{w_i^-\, : \, i\geq 1\}$ are basis of $W^+$ and $W^-$, respectively,  such that $w_0^+=1_W$,
 $$\Phi(w_1^+)=(R_{e_{11}+e_{44}-e_{22}-e_{33}},L_{e_{11}+e_{44}-e_{22}-e_{33}})\in \M(M) \qquad \mbox{and} \qquad w_i^+,\, w_j^-\in \ker \Phi \quad  \mbox{for all} \ i \geq 2, \, j\geq 1.$$ 
    Let $I$ be the $T_W^*$-ideal generated by the $(W,*)$-polynomials in the statement. It is easy to check that
            $I \subseteq \I^{W,*}(M^D)$. 
            
            To prove the opposite inclusion, let $f$ be a multilinear generalized $*$-polynomial of degree $n$ and assume, as we may, that $f\in P_{r,n-r}$ for some $0\leq r \leq n$.

            Let $\epsilon\in \{+,-\}$. Observe first that $x^\epsilon (w_1^+)^2 - x^\epsilon \in \langle (w_1^+)^2x^\epsilon - x^\epsilon\rangle_{T_W^*}$. Moreover, for $i\geq 2$ we have $x^{\epsilon}w_i^+, \, w_1^+w_i^+x^\epsilon, \, x^{\epsilon}w_i^+w_1^+ \in \langle w_i^+x^\epsilon\rangle_{T_W^*}$ and for $j\geq 1$ we have $x^{\epsilon}w_j^-, \, w_1^+w_j^-x^\epsilon, \, x^{\epsilon}w_j^-w_1^+\in \langle w_j^-x^\epsilon\rangle_{T_W^*}$. Furthermore since the $(W,*)$-polynomial $w_1^+ x^{\epsilon}+ x^{\epsilon}w_1^+$ is symmetric if $\epsilon = +$ and skew-symmetric if $\epsilon = -$, it follows that $x^{\epsilon}w_1^+w_i^+, \, w_i^+w_1^+x^\epsilon\in \langle w_i^+x^\epsilon\rangle_{T_W^*}$ for all $i\geq 2$, and $x^{\epsilon}w_1^+w_j^-, \, w_j^-w_1^+ x^\epsilon\in \langle w_j^-x^\epsilon\rangle_{T_W^*}$ for all $j\geq 1$. Hence, since $(w_1^+)^2 x^\epsilon - x^\epsilon,
            \,  w_i^+x^\epsilon, \, w_j^-x^\epsilon\in I$, for $\epsilon\in \{+,-\}$, $i\geq 2$ and $j\geq 1$,  we may suppose that, modulo $I$,  only $1_W$ and $w_1^+$ appear in $f$.

            Since $  [x^+_1, w_1^+]\in W\langle X, *\rangle^-$, by Lemma \ref{Lem: consequence x_1^- x_2^-} we get that for every $k\geq 0$
    \begin{equation*}
        [x^+_{k+1}, w_1^+] x_1^+ \cdots x_k^+ x_{k+2}^-, \quad  x^-_{k+2} x_1^+ \cdots x_k^+ [x_{k+1}^+,w_1^+]  \;\in I.
    \end{equation*} 
      Hence, using $ w_1^+x^-+x^-w_1^+, \, (w_1^+)^2 x^\epsilon - x^\epsilon, \, x^\epsilon (w_1^+)^2 - x^\epsilon\in I$, with $\epsilon\in \{ +,-\}$, we obtain for every $k\geq 0$
    \begin{equation}\label{eq: consequence w_1 1}
          v_{0}x_1^+ v_{1}x_2^+ \cdots v_{k-1} x_k^+ v_{k} x_{k+1}^- v_{k+1} \equiv x_1^+\cdots  x_k^+  x_{k+1}^- v \pmod{I}
    \end{equation}
    and
    \begin{equation}\label{eq: consequence w_1 2}
          v_{0} x_{k+1}^- v_{1} x_1^+ v_{2}x_2^+\cdots v_{k} x_k^+ v_{k+1}   \equiv x_{k+1}^-  x_1^+\cdots  x_k^+ v \pmod{I}
    \end{equation}
    where  $v, v_{0}, \ldots, v_{k+1}\in \{1_W, w_1^+\}$.

  Now, combining Lemma \ref{Lem: consequence x_1^- x_2^-} with \eqref{eq: consequence w_1 1}, we get for every $k\geq 0$
    \begin{equation}\label{eq: consequences}
        u_1x_{k+1}^- v_{0}x_1^+ v_{1}x_2^+\cdots v_{k-1} x_k^+ v_{k} x_{k+2}^- u_2\in I,
    \end{equation}
where  $u_1, u_2, v_{0}, \ldots, v_{k+1}\in \{1_W, w_1^+\}$. Thus, it follows that, modulo $I$, every monomial appearing in $f$ contains at most one variable of type $x^-$. In particular, either  $f\in P_{n-1,1}$ or $f\in P_{n,0}.$

 	Let us first consider $f\in P_{n-1,1}$. From  \eqref{eq: consequence w_1 1} and  \eqref{eq: consequence w_1 2} it follows that we can move $w_1^+$ to the end of each monomial of $f$. Moreover, by Lemma \ref{Lem: consequence x_1^- x_2^-} (see \cite[Remark 2]{MisVal}), in every monomial $x_{i_1}^+\cdots x_{i_{k}}^+x_{n}^-x_{j_{1}}^+ \cdots x_{j_{n-1-k}}^+$ we can reorder the variables to the left and to the right of $x_{n}^-$.
    Thus, we can write $f$, modulo $I$, as linear combination of the $(W,*)$-polynomials
    \begin{equation}\label{eq: generators P_n-1,1}
        x_{i_1}^+\cdots x_{i_{k}}^+x_{n}^-x_{j_{1}}^+ \cdots x_{j_{n-k-1}}^+, \quad x_{i_1}^+\cdots x_{i_{k}}^+x_{n}^-x_{j_{1}}^+ \cdots x_{j_{n-k-1}}^+w_1^+
    \end{equation}
     where $0\leq k \leq n-1$, $i_1< \ldots < i_k$ and $j_1< \ldots < j_{n-k-1}$.

Next, we prove that $(W,*)$-polynomials in \eqref{eq: generators P_n-1,1} are linearly independent modulo $\I^{W,*}(M^D).$ To this end, let us assume that $f\in \I^{W,*}(M^D)$ is a linear combination of the above $(W,*)$-polynomials, i.e.,
$$
f=\sum_{\mathcal{I}, \mathcal{J}} \alpha_{\mathcal{I}, \mathcal{J}} x_{i_1}^+\cdots x_{i_{k}}^+x_{n}^-x_{j_{1}}^+ \cdots x_{j_{n-k-1}}^+ + \sum_{\mathcal{I}, \mathcal{J}} \beta_{\mathcal{I}, \mathcal{J}} x_{i_1}^+\cdots x_{i_{k}}^+x_{n}^-x_{j_{1}}^+ \cdots x_{j_{n-k-1}}^+w_1^+
$$
where $\mathcal{I}=\{i_1, \ldots, i_k\}$ and $\mathcal{J}=\{j_1,\ldots, j_{n-k-1}\}$ are subsets of indices in $\{1,\ldots, n-1\}$ such that $\mathcal{I}\cap \mathcal{J}=\emptyset$ with $0\leq k \leq n-1$, $i_1< \cdots < i_k$, $j_1< \cdots < j_{n-k-1}$, and $\alpha_{\mathcal{I},\mathcal{J}}, \beta_{\mathcal{I},\mathcal{J}}\in F$ for all $\mathcal{I}, \mathcal{J}$.

Let $\mathcal{I}=\{i_1, \ldots, i_k\}$ and $\mathcal{J}=\{j_1,\ldots, j_{n-k-1}\}$ be two fixed disjoint subsets of $\{1,\ldots, n-1\}$ such that $0<k < n-1$, $i_1< \cdots < i_k$, $j_1< \cdots < j_{n-k-1}$.
If we make the evaluation $ x_{i_1}^+=\cdots =x_{i_{k}}^+= e_{11}+e_{44}$, $x_{j_{1}}^+ =\cdots= x_{j_{n-k-1}}^+= e_{22}+e_{33}$, and $x_n^-=e_{12}-e_{34}$, we get $(\alpha_{\mathcal{I}, \mathcal{J}} - \beta_{\mathcal{I}, \mathcal{J}})e_{12}-(\alpha_{\mathcal{J}, \mathcal{I}} + \beta_{\mathcal{J}, \mathcal{I}}) e_{34} =0$,  that is
\begin{equation}\label{eq: evaluation 1}
    \alpha_{\mathcal{I}, \mathcal{J}} - \beta_{\mathcal{I}, \mathcal{J}}=0 \quad \text{and} \quad \alpha_{\mathcal{J}, \mathcal{I}} + \beta_{\mathcal{J}, \mathcal{I}}=0.
\end{equation}
Moreover, let us make the evaluation $ x_{i_1}^+=\cdots =x_{i_{k}}^+= e_{22}+e_{33}$, $x_{j_{1}}^+ =\cdots= x_{j_{n-k-1}}^+= e_{11}+e_{44}$, and $x_n^-=e_{12}-e_{34}$. In this case we get $-(\alpha_{\mathcal{I}, \mathcal{J}} + \beta_{\mathcal{I}, \mathcal{J}})e_{34}+(\alpha_{\mathcal{J}, \mathcal{I}} - \beta_{\mathcal{J}, \mathcal{I}}) e_{12} =0$,  thus
\begin{equation}\label{eq: evaluation 2}
   \alpha_{\mathcal{I}, \mathcal{J}} + \beta_{\mathcal{I}, \mathcal{J}}=0 \quad \text{and} \quad \alpha_{\mathcal{J}, \mathcal{I}} - \beta_{\mathcal{J}, \mathcal{I}}=0.
\end{equation}
Hence, by putting together \eqref{eq: evaluation 1} and \eqref{eq: evaluation 2} we obtain that $ \alpha_{\mathcal{I}, \mathcal{J}} = \beta_{\mathcal{I}, \mathcal{J}}=\alpha_{\mathcal{J}, \mathcal{I}} = \beta_{\mathcal{J}, \mathcal{I}}=0$.

Finally, fixed $\mathcal{I}=\{1,\ldots,n-1\}$, if we evaluate $x_1=\cdots = x_{n-1}=e_{11}+e_{44}$ and $x_n^-=e_{12}-e_{34}$, then we get $(\alpha_{\mathcal{I}, \emptyset} - \beta_{\mathcal{I}, \emptyset})e_{12}-(\alpha_{\emptyset, \mathcal{I}} + \beta_{\emptyset, \mathcal{I}}) e_{34} =0$,  thus
\begin{equation}\label{eq: evaluation 3}
    \alpha_{\mathcal{I}, \emptyset} - \beta_{\mathcal{I}, \emptyset}=0 \quad \text{and} \quad \alpha_{\emptyset, \mathcal{I}} + \beta_{\emptyset, \mathcal{I}}=0.
\end{equation}
Moreover, from the evaluation $x_1=\cdots = x_{n-1}=e_{22}+e_{33}$ and $x_n^-=e_{12}-e_{34}$, it follows that $-(\alpha_{\mathcal{I}, \emptyset} + \beta_{\mathcal{I}, \emptyset})e_{34}+(\alpha_{\emptyset, \mathcal{I}} - \beta_{\emptyset, \mathcal{I}}) e_{12} =0$,  thus
\begin{equation}\label{eq: evaluation 4}
   \alpha_{\mathcal{I}, \emptyset} + \beta_{\mathcal{I}, \emptyset}=0 \quad \text{and} \quad \alpha_{\emptyset, \mathcal{I}} - \beta_{\emptyset, \mathcal{I}}=0.
\end{equation}
Therefore, putting together \eqref{eq: evaluation 3} and \eqref{eq: evaluation 4} we get  $\alpha_{\mathcal{I}, \emptyset}= \beta_{\mathcal{I}, \emptyset}=\alpha_{\emptyset, \mathcal{I}} = \beta_{\emptyset, \mathcal{I}}=0$. As a consequence, all the scalars appearing in $f$ are zero.

Thus, the polynomials in \eqref{eq: generators P_n-1,1} are linearly independent modulo $\I^{W,*}(M^D)$ and, since $P_{n-1,1}\cap I\subseteq P_{n-1,1}\cap \I^{W,*}(M^D),$ this proves that
the polynomials in \eqref{eq: generators P_n-1,1} are a basis of $P_{n-1,1}$ modulo $P_{n-1,1}\cap \I^{W,*}(M^D).$ Hence, by counting we get
\begin{equation}\label{eq: c_{n-1,1}(M^D)}
    c_{n-1,1}(M^D)= 2 \sum_{k=0}^{n-1}\binom{n-1}{k}=2^n.
\end{equation}

Now, let us consider $f\in P_{n,0}$. Since $[x_1^+,x_2^+], [w_1^+, x_1^+] \in W\langle X,* \rangle^-$, from \eqref{eq: consequence w_1 1} we obtain for every $k\geq 0$
    \begin{align}\label{eq: consequence w_1 3}
          v_{0}x_1^+ v_{1}x_2^+\cdots v_{k-1} x_k^+ v_{k} [x_{k+1}^+,x_{k+2}^+] v_{k+1} &\equiv x_1^+\cdots  x_k^+  [x_{k+1}^+,x_{k+2}^+]  v \pmod{I},\\
          \label{eq: consequence w_1 3bis}
 v_{0}x_1^+ v_{1}x_2^+\cdots v_{k-1} x_k^+ v_{k} [w_1^+, x_{k+1}^+] v_{k+1} &\equiv x_1^+\cdots  x_k^+  [w_1^+, x_{k+1}^+]  v \pmod{I},
 \end{align}
 and from \eqref{eq: consequence w_1 2} we get for every $k\geq 0$
    \begin{align}
          \label{eq: consequence w_1 4}
          v_{0} [x_{k+1}^+,x_{k+2}^+] v_{1} x_1^+ v_{2}x_2^+\cdots v_{k} x_k^+ v_{k+1}   &\equiv [x_{k+1}^+,x_{k+2}^+]  x_1^+\cdots  x_k^+ v \pmod{I},\\
          \label{eq: consequence w_1 4 bis}
            v_{0} [w_1^+, x_{k+1}^+] v_{1} x_1^+ v_{2}x_2^+\cdots v_{k} x_k^+ v_{k+1}   &\equiv [w_1^+, x_{k+1}^+]  x_1^+\cdots  x_k^+ v \pmod{I}.
    \end{align}
    where  $v, v_{0}, \ldots, v_{k+1}\in \{1_W, w_1^+\}$. As a consequence, by the Poincaré-Birkhoff-Witt Theorem, $f$ can be written, modulo $I$, as a linear combination of polynomials of the type
    \begin{equation}\label{eq: PBW}
        x_{i_1}^+\cdots x_{i_m}^+c_1c_2\cdots c_s v,
    \end{equation}
    where $i_1<\cdots < i_m,$ $c_1,\ldots, c_s$ are left-normed commutators containing symmetric variables and $w_1^+$, and $v\in \{1_W, w_1^+\}$.

     Now, by  \eqref{eq: consequences} and again since $[x_1^+,x_2^+], [x_1^+,w_1^+]\in W\langle X,* \rangle^-$, we get for every $k\geq 0$
     \begin{align*}
          u_1[x_{k+1}^+,x_{k+2}^+] v_{0}x_1^+ v_{1}x_2^+\cdots v_{k-1} x_k^+ v_{k} [x_{k+3}^+,x_{k+4}^+]u_2 \, &\in I,\\
          u_1[w_1^+, x_{k+1}^+] v_{0}x_1^+ v_{1}x_2^+\cdots v_{k-1} x_k^+ v_{k} [x_{k+2}^+,x_{k+3}^+]u_2 \, &\in I,\\
          u_1[x_{k+1}^+,x_{k+2}^+] v_{0}x_1^+ v_{1}x_2^+\cdots v_{k-1} x_k^+ v_{k} [w_1^+, x_{k+3}^+]u_2 \, &\in I,\\
           u_1[w_1^+, x_{k+1}^+] v_{0}x_1^+ v_{1}x_2^+\cdots v_{k-1} x_k^+ v_{k} [w_1^+, x_{k+2}^+]u_2 \, &\in I,
          \end{align*}
where  $u_1, u_2, v_{0}, \ldots, v_{k}\in \{1_W, w_1^+\}$. Thus it follows that in \eqref{eq: PBW} only a left-normed commutator can appear, i.e., $s\in \{0,1\}$. Moreover, by \eqref{eq: consequence w_1 3}--\eqref{eq: consequence w_1 4 bis}, we can assume that if $w_1^+$ appears in the commutator $c_1$, then it is in the first position. Hence, $f$ can be written, modulo $I$, as a linear combination of $(W,*)$-monomials $x_1\cdots x_n v$ and the $(W,*)$-polynomials of the type
\begin{equation*}\label{eq: almost a basis}
    x_{i_1}^+\cdots x_{i_{m}}^+[x_{k}^+,x_{j_1}^+,\ldots, x_{j_{n-m-1}}^+]v, \quad x_{i_1}^+\cdots x_{i_{m}}^+[w_1^+, x_{l_1}^+,\ldots, x_{l_{n-m}}^+]v
\end{equation*}
 where $0\leq m \leq n-2$, $i_1<\cdots < i_m$ and $v\in \{1_W, w_1^+\}$.  Since $[x_{k}^+, x_{j_1}^+], [w_1^+,x_{l_1}^+]\in W \langle X,*\rangle^-$, Lemma \ref{Lem: consequence x_1^- x_2^-} implies that we can reorder the variables to the left and to the right of $[x_{k}^+,x_{j_1}^+]$ and $[w_1^+,x_{l_1}^+]$ (see \cite[Remark 2]{MisVal}). Hence we may assume that $j_2<  \cdots < j_{n-m-1}$ and  $l_2<  \cdots < l_{n-m}$.
 Moreover, by the Jacobi identity we may also suppose that $k>j_1<j_2$.  Additionally, applying the Jacobi identity, \eqref{eq: consequence w_1 3} and \eqref{eq: consequence w_1 4}, we may suppose that $l_1<l_2$. This implies that we can write $f$, modulo $I$, as linear combination of the $(W,*)$-polynomials
 \begin{equation}\label{eq: generators P_n,0}
 \begin{split}
      &  x_{i_1}^+\cdots x_{i_{m}}^+[x_{k}^+,x_{j_1}^+,\ldots, x_{j_{n-m-1}}^+], \quad  x_{i_1}^+\cdots x_{i_{m}}^+[x_{k}^+,x_{j_1}^+,\ldots, x_{j_{n-m-1}}^+]w_1^+ \\
  &  x_{h_1}^+\cdots x_{h_{p}}^+[w_1^+, x_{l_1}^+,\ldots, x_{l_{n-p}}^+], \quad  x_{h_1}^+\cdots x_{h_{p}}^+[w_1^+, x_{l_1}^+,\ldots, x_{l_{n-p}}^+]w_1^+,
 \end{split}
\end{equation}
where $0\leq m \leq n$, $m\neq n-1$, $0\leq p\leq n-1$, $i_1 < \cdots < i_m$, $k>j_1< \cdots < j_{n-m-1}$, $h_1< \cdots < h_p$, $l_1< \cdots < l_{n-p}$.

Next, we prove that these $(W,*)$-polynomials are linearly independent modulo $\I^{W,*}(M^D)$. To this end, let us assume $f\in \I^{W,*}(M^D)$ is such that
\begin{equation}\label{eq: id}
\begin{split}
     f= & \sum_{\mathcal{I}, k} \alpha_{\mathcal{I}, k} x_{i_1}^+\cdots x_{i_{m}}^+[x_{k}^+,x_{j_1}^+,\ldots, x_{j_{n-m-1}}^+] + \sum_{\mathcal{I}, k} \beta_{\mathcal{I}, k} x_{i_1}^+\cdots x_{i_{m}}^+[x_{k}^+,x_{j_1}^+,\ldots, x_{j_{n-m-1}}^+]w_1^+ \\
  & +\sum_{\mathcal{H}} \gamma_{\mathcal{H}} x_{h_1}^+\cdots x_{h_{p}}^+[w_1^+, x_{l_1}^+,\ldots, x_{l_{n-p}}^+]+ \sum_{\mathcal{H}} \delta_{\mathcal{H}} x_{h_1}^+\cdots x_{h_{p}}^+[w_1^+, x_{l_1}^+,\ldots, x_{l_{n-p}}^+]w_1^+,
\end{split}
\end{equation}
where $0\leq m \leq n$, $m\neq n-1$, $0\leq p\leq n-1$, $i_1< \cdots < i_m$, $k>j_1< \cdots < j_{n-m-1}$, $h_1< \cdots < h_p$, $l_1<\cdots < l_{n-p}$,  $\mathcal{I}=\{i_1, \ldots, i_m\}$, $\mathcal{H}=\{h_1,\ldots, h_p\}$ and $\alpha_{\mathcal{I}, k}, \beta_{\mathcal{I},k},\gamma_{\mathcal{H}}, \delta_{\mathcal{H}} \in F$ for all $\mathcal{I}, \mathcal{H}$.

First, suppose that $\alpha_{\mathcal{I},0} \neq 0$ or $\beta_{\mathcal{I},0} \neq 0$ in \eqref{eq: id}  with  $ \mathcal{I}=\{1,\ldots, n\}.$ From the evaluation $x_1^+=\cdots = x_n^+=e_{11}+e_{22}+e_{33}+e_{44}$ it follows that $\alpha_{\mathcal{I},0} (e_{11}+e_{22}+e_{33}+e_{44})+ \beta_{\mathcal{I},0}(e_{11}+e_{44}-e_{22}-e_{33})=0 $, thus $\alpha_{\mathcal{I},0} =\beta_{\mathcal{I},0}=0$, a contradiction. This says that in \eqref{eq: id} only terms with one commutator can appear with a non-zero coefficient.

Now, let $k\in \{1,\ldots, n\}$ and let $\mathcal{I}=\{i_1, \ldots, i_m\}$ be a proper subset of $\{1,\ldots, n\}\setminus\{k\}$ where $i_1<\cdots <i_m$ and $ 0 \leq m\leq n-2$ is the maximum such that $\alpha_{\mathcal{I},k} \neq 0$ or $\beta_{\mathcal{I},k}\neq 0$ in \eqref{eq: id}. If we consider the evaluation $x_{i_1}^+=\cdots =x_{i_m}^+=e_{11}+e_{22}+e_{33}+e_{44},$
$ x_{k}^+= e_{12}+e_{34},$  and $x_{j_1}^+=\cdots = x_{j_{n-m-1}}^+=e_{22}+e_{33},$ then the third and fourth summations in \eqref{eq: id} are evaluated to zero since $\Phi(w_1^+)(e_{22}+e_{33})- (e_{22}+e_{33})\Phi(w_1^+)=0$, and we get $$(\alpha_{\mathcal{I}, k} -  \beta_{\mathcal{I},k})e_{12} + (-1)^{n-m-1} (\alpha_{\mathcal{I}, k} +  \beta_{\mathcal{I},k}) e_{34} =0.$$
Thus $ \alpha_{\mathcal{I}, k} -  \beta_{\mathcal{I},k}=0$ and $\alpha_{\mathcal{I}, k} +  \beta_{\mathcal{I},k}=0.$
As a consequence, it follows that $\alpha_{\mathcal{I}, k}= \beta_{\mathcal{I},k}=0$, a contradiction. By repeating this argument we obtain that $\alpha_{\mathcal{I}, k}= \beta_{\mathcal{I},k}=0$ for all $\mathcal{I}$ and $k$.

Now suppose that $\mathcal{H}=\{h_1, \ldots, h_p\}$ is a proper subset of $\{1,\ldots, n\}$ where $h_1<\cdots <h_p$ and $ 0 \leq p\leq n-2$ is the maximum such that $\gamma_{\mathcal{H}} \neq 0$ or $\delta_{\mathcal{H}}\neq 0$ in \eqref{eq: id}, and consider the evaluation $x_{h_1}^+=\cdots =x_{h_p}^+=e_{11}+e_{22}+e_{33}+e_{44},$
$ x_{l_1}^+= e_{12}+e_{34},$  and $x_{l_2}^+=\cdots = x_{l_{n-p}}^+=e_{22}+e_{33}.$ Under this evaluation, we obtain
$$2(\gamma_{\mathcal{H}} -  \delta_{\mathcal{H}})e_{12} + 2(-1)^{n-p-1} (\gamma_{\mathcal{H}} +  \delta_{\mathcal{H}}) e_{34} =0.$$
As a consequence, since $\cha F=0$, we get that $\gamma_{\mathcal{H}}=  \delta_{\mathcal{H}}=0$, a contradiction. As above, by repeating this argument we obtain that $\gamma_{\mathcal{H}}=  \delta_{\mathcal{H}}=0$ for all $\mathcal{H}$. As a consequence, all the scalars appearing in $f$ are zero, i.e., $f=0.$

Thus, the polynomials in \eqref{eq: generators P_n,0} are linearly independent modulo $\I^{W,*}(M^D)$ and,
as in the previous case, this implies that the $(W,*)$-polynomials in \eqref{eq: generators P_n,0} are a basis of $P_{n,0}$ modulo $P_{n,0}\cap \I^{W,*}(M^D).$ Hence, by counting we get
\begin{equation}\label{eq: c_{n,0}(M^D)}
    c_{n,0}(M^D)= 2+2\sum_{m=2}^n\binom{n}{m}(m-1)+ 2\sum_{m=0}^{n-1}\binom{n}{m}=2\sum_{m=0}^n\binom{n}{m}m +2= n2^n +2.
\end{equation}

Since we have shown that $P_{r,n-r}\cap \I^{W,*}(M^D) \supseteq P_{r,n-r}\cap I$  for all $n\in\N$ and $0\leq r\leq n$, it follows that $\I^{W,*}(M^D)=I$, as claimed.

Finally, since $P_{r,n-r}\equiv 0 \pmod{(P_{r,n-r}\cap \I^{W,*}(M^D))}$ for $0\leq r \leq n-2$, applying Formula~\eqref{eq: decomp codim} together with \eqref{eq: c_{n-1,1}(M^D)} and \eqref{eq: c_{n,0}(M^D)}, we obtain
\begin{align*}
 	c_n^{W,*}(M^D)=\sum_{r=0}^n\binom{n}{r}c_{r,n-r}^{W,*}(M^D)=c_{n,0}^{W,*}(M^D)+ n c_{n-1,1}^{W,*}(M^D)=n2^{n}+2+ n  2^{n}= n2^{n+1}  +2,
 	\end{align*}
    as desired.
 	\end{proof}

\begin{theorem}\label{thm: id e cn M^F}
    Let $M^F$ be a $(W,*)$-algebra $M$ where $W$ acts on it as the algebra $F$ by left and right multiplication.
     Then there exists a basis $\mathcal{B}_W=\mathcal{B}_{W^+}\cup \mathcal{B}_{W^-}$  of $W$, where $\mathcal{B}_{W^+}:=\{1_W, \, w_i^+\, : \, i\geq 1\}$ and  $\mathcal{B}_{W^-}:=\{w_i^-\, : \, i\geq 1\}$ are bases of $W^+$ and $W^-$, respectively, such that 
   the $T_W^*$-ideal $\I^{W,*}(M^F)$  of generalized $*$-identities of $M^F$ is generated, as  $T_W^*$-ideal, by the following polynomials: 
        $$
       x_1^-x_2^-, \quad w_i^\eta x^\epsilon,
        $$
        where $\epsilon,\eta\in \{+,-\}$ and  $i\geq 1$. Moreover $c_n^{W,*}(M^F)= (n-1)2^{n}+2$.
\end{theorem}
\begin{proof}
Let $\Phi$ denote the acting homomorphism of $W$ on $M^F$. Then $\Phi(W)\cong F$ and we can choose a basis $\mathcal{B}_W=\mathcal{B}_{W^+}\cup \mathcal{B}_{W^-}$  of $W=W^+\oplus W^-$, where $\mathcal{B}_{W^+}:=\{ w_i^+\, : \, i\geq 0\}$ and  $\mathcal{B}_{W^-}:=\{w_i^-\, : \, i\geq 1\}$ are bases of $W^+$ and $W^-$, respectively,  such that $w_0^+=1_W$, and 
 $ w_i^+,w_i^-\in \ker \Phi$ for all $ i \geq 1.$ 
      Clearly, the $(W,*)$-identities of $M^F$ correspond to the ordinary $*$-identities of $M$ modulo the $T_W^*$-ideal $\langle w_i^{\eta} x^\epsilon   :  \epsilon,\eta\in \{+,-\}, \, i\geq 1\rangle_{T_W^*}$. Thus by \cite[Theorem 2]{MisVal} we obtain the required conclusion:
      $$\I^{W,*}(M^F)=\langle x_1^-x_2^-, \, w_i^{\eta} x^\epsilon : \epsilon,\eta\in \{+,-\}, \, i\geq 1\rangle_{T_W^*}.
      $$

      Moreover, by Lemma~\ref{Lem: consequence x_1^- x_2^-} and arguing along the same lines as in the proof of Theorem~\ref{thm: id e cn M^D}, with the necessary changes, we obtain the following.

      The $(W,*)$-polynomials
     \begin{equation*}
  x_{i_1}^+\cdots x_{i_{m}}^+[x_{k}^+,x_{j_1}^+,\ldots, x_{j_{n-m-1}}^+],
\end{equation*}
where $0\leq m \leq n$, $m\neq n-1$, $i_1 < \cdots < i_m$, $k>j_1< \cdots < j_{n-m-1}$, and the $(W,*)$-polynomials
  \begin{equation*}
        x_{i_1}^+\cdots x_{i_{k}}^+x_{n}^-x_{j_{1}}^+ \cdots x_{j_{n-k-1}}^+,
    \end{equation*}
     where $0\leq k \leq n-1$, $i_1< \cdots < i_k$, $j_1< \cdots < j_{n-k-1}$, form a basis of $P_{r,n-r}$ modulo $P_{r,n-r}\cap \I^{W,*}(M^F)$ for $r=n$ and  $r=n-1$, respectively. Moreover, $P_{r,n-r}\equiv 0  \pmod{(P_{r,n-r}\cap \I^{W,*}(M^F))}$ for $0\leq r \leq n-2$. Consequently, a direct counting argument yields
$$c_{r,n-r}(M^F)=\begin{cases}
(n-2)2^{n-1} +2,  & \mbox{ if } r=n, \\
2^{n-1},  & \mbox{ if } r=n-1, \\
0, & \mbox{ if } 0\leq r\leq n-2.
\end{cases}$$
Finally, by Formula~\eqref{eq: decomp codim}, we conclude that $c_n^{W,*}(M^F)= (n-1)2^{n}+2$, which completes the proof.
\end{proof}

Recall that $M^D$ and $M^F$ denote fixed representatives of two distinct
equivalence classes in the sense of Definition~\eqref{def: equivalence}.
More precisely, there exist acting $*$-homomorphisms $\Phi_D\colon W\to\M(M)$ and $\Phi_F\colon W \to \M(M)$ with images
$\Phi_D(W)\cong D$ and $\Phi_F(W)\cong F$ such that the $(W,*)$-algebras $M^D$ and $M^F$ are precisely
$M^{\Phi_D}$ and $M^{\Phi_F}$, respectively. Therefore, to prove that $M^D$ and $M^F$ generate distinct
$(W,*)$-varieties, it suffices to show that, for every choice of acting
$*$-homomorphisms $\Phi_D$ and $\Phi_F$ of $W$ on $M$ with
$\Phi_D(W)\cong D$ and $\Phi_F(W)\cong F$, the $(W,*)$-algebras $M^{\Phi_D}$ and $M^{\Phi_F}$ generate distinct $(W,*)$-varieties.

\begin{corollary}\label{cor: mf md distinct}
    The $(W,*)$-algebras $M^F$ and $M^D$ generates two distinct $(W,*)$-varieties.
\end{corollary}
\begin{proof}
Given an acting $*$-homomorphism $\Phi_D$  of $W$ on $M$ such that $\Phi_D(W)\cong D$,
    by Theorem \ref{thm: id e cn M^D}  there exists a symmetric element  $w_1^+$ of $W$ linearly independent with $1_W$ such that $(w_1^+)^2x^+-x^+\in \I^{W,*}(M^{\Phi_D})$.
    For any acting $*$-homomorphism $\Phi_F$  of $W$ on $M$ such that $\Phi_F(W)\cong F$, we have that $\Phi_F(w_1^+)=(0,0)\in \M(M)$.  As a consequence, $(w_1^+)^2x^+-x^+\notin \I^{W,*}(M^F)$; indeed, $(w_1^+)^2 1_M - 1_M=-1_M\neq 0$. Therefore, $\I^{W,*}(M^{\Phi_D})\not\subset \I^{W,*}(M^{\Phi_F})$ for any acting
$*$-homomorphisms $\Phi_D$ and $\Phi_F$ of $W$ on $M$ with
$\Phi_D(W)\cong D$ and $\Phi_F(W)\cong F$, and consequentially $M^F\notin \V^{W,*}(M^D).$

Now, given an acting $*$-homomorphism $\Phi_F$  of $W$ on $M$ such that $\Phi_F(W)\cong F$,
 by Theorem \ref{thm: id e cn M^F} for any element $w$ of $W$ linearly independent with $1_W$  we have that $wx^+\in \I^{W,*}(M^{\Phi_F})$.
 For any acting $*$-homomorphism $\Phi_D$  of $W$ on $M$ such that $\Phi_D(W)\cong D$, there exists a symmetric element  $w_1^+$ of $W$ linearly independent with $1_W$ such that $\Phi_D(w_1^+)=(R_{e_{11}+e_{44}-e_{22}-e_{33}},L_{e_{11}+e_{44}-e_{22}-e_{33}})\in \M(M)$. As a consequence, $w_1^+x^+\notin \I^{W,*}(M^{\Phi_D})$ since $w_1^+1_M=e_{11}+e_{44}-e_{22}-e_{33}\neq 0 $. Thus, $\I^{W,*}(M^{\Phi_D})\not\subset \I^{W,*}(M^{\Phi_F})$ for any acting
$*$-homomorphisms $\Phi_D$ and $\Phi_F$ of $W$ on $M$ with
$\Phi_D(W)\cong D$ and $\Phi_F(W)\cong F$, and as a consequence $M^D\notin \V^{W,*}(M^F).$
\end{proof}

\section{Generalized $*$-cocharacter of $M^D$}

In this section, we compute the $(r,n-r)$th $(W,*)$-cocharacter of $M^D$, which will play a central role in the next section in studying the growth of the $(W,*)$-variety generated by $M^D$. More precisely, for $n \geq 1$ and $0 \leq r \leq n$, we determine the multiplicities $m_{\lambda,\mu}$ in the decomposition of the $(S_r \times S_{n-r})$-character
\begin{equation}\label{eq: ch M^D}
        \chi_{r,n-r}^{W,*}(M^D)=\sum_{(\lambda,\mu)\vdash (r,n-r)}m_{\lambda,\mu}(\chi_\lambda \otimes \chi_\mu)
\end{equation}
of $P_{r,n-r}(M^D)$ (see \eqref{eq: ch}). 
To this end recall that any irreducible left $(S_r\times S_{n-r})$-module $W_{\lambda,\mu}\subseteq P_{r,n-r}$ with character $\chi_\lambda \otimes \chi_\mu$ can be generated as an $(S_r\times S_{n-r})$-module by an element of the form $e_{T_\lambda}e_{T_\mu}f$, for some $f\in W_{\lambda,\mu}$ and
some pair of Young tableaux $(T_\lambda,T_\mu)$ of shape $\lambda\vdash r$ and $\mu\vdash n-r$, respectively. Here $e_{T_{\nu}}=\sum_{\substack{\sigma\in R_{T_{\nu}} \\ \tau\in C_{T_{\nu}}}} (\sgn\tau)\sigma\tau$  stands for the quasi-idempotent corresponding to
some Young tableau $T_\nu$ of shape $\nu \vdash t$, where $R_{T_{\nu}}$ and $C_{T_{\nu}}$ are the subgroups of $S_t$ stabilizing the rows and columns of $T_{\nu}$, respectively.
Furthermore, recall that 
\begin{equation}\label{eq: codim and ch M^D}
    c_{r,n-r}^{W,*}(M^D)=\sum_{(\lambda, \mu)\vdash (r,n-r)}m_{\lambda, \mu}d_{\lambda}d_{\mu}
\end{equation}
where $d_{\lambda}=\chi_{\lambda}(1)$ and $d_{\mu}=\chi_{\mu}(1)$ are the degrees of $\chi_\lambda$ and $\chi_\mu$, i.e., the number of standard tableaux of shape $\lambda$ and $\mu,$ respectively.

We start by proving the following lemmas.

 \begin{lemma}\label{lemma: cn-1,1}
 If $r=n-1$, then in \eqref{eq: ch M^D} we have 
$$
m_{\lambda,\mu}= \begin{cases}
             2(q+1)  & \text{if } \lambda=(p+q,p),  \mu=(1)\\
              2n & \text{if }  \lambda=(n-1),  \mu=(1)\\
            0 & \text{otherwise}
    \end{cases}.
$$
\end{lemma}
\begin{proof}
Notice that, as a consequence of formula \eqref{eq: codim and ch M^D}, it suffices to prove that:
\begin{enumerate}
    \item $m_{(p+q,p),(1)}\geq 2(q+1)$ for all $p \geq 1$ and $q \geq 0$ such that $2p+q=n-1$;
    \vspace{1mm}
    \item $m_{(n-1),(1)}\geq 2n$.
\end{enumerate}
Indeed,
assume that $m_{(n-1),(1)}\geq 2n$ and $m_{(p+q,p),(1)}\geq 2(q+1)$ for all $p \geq 1$ and $q \geq 0$ such that $2p+q=n-1$. Then
\begin{equation}\label{eq: c_{n-1,1}(M^D) diseq}
    c_{n-1,1}(M^D)\geq\sum_{\substack{1\leq p \leq \lfloor \frac{n-1}{2}\rfloor\\0 \leq q \leq n-2p-1}} 2(q+1) d_{(p+q,p)} d_{(1)}+2n d_{(n-1)}d_{(1)}.
\end{equation}
Now, by the formula $(4.8)$ proved in Theorem 4.4 of \cite{MarRiz2}, we have
\begin{equation*}\label{eq: d_(p+q,p)}
    \sum_{\substack{1\leq p \leq \lfloor \frac{n-1}{2}\rfloor\\0 \leq q \leq n-2p-1}} 2(q+1) d_{(p+q,p)}=2  (2^{n-1}-n)=2^n-2n.
\end{equation*}
Hence, recalling that $d_{(n-1)}=d_{(1)}=1$, we get
$$\sum_{\substack{1\leq p \leq \lfloor \frac{n-1}{2}\rfloor\\0 \leq q \leq n-2p-1}} 2(q+1) d_{(p+q,p)} d_{(1)}+2n  d_{(n-1)}d_{(1)}=2^n-2n+2n=2^n.$$
Since by \eqref{eq: c_{n-1,1}(M^D)} $c_{n-1,1}(M^D)=2^n$, it follows that \eqref{eq: c_{n-1,1}(M^D) diseq} is actually an equality. Thus the lemma follows once $(1)$ and $(2)$ are proved. 

\vspace{1ex}
We compute the two different cases separately.

       \vspace{1ex} 
\noindent\textbf{Case 1:} We prove that $m_{(p+q,p),(1)}\geq 2(q+1)$ for all $p \geq 1$, $q \geq 0$ such that $2p+q=n-1$.
		
		For $0\leq i \leq q$, we consider the following two tableaux:
         $$
			T_{(p+q,p)}^{(i)}=	\begin{array}{|c|c|c|c|c|c|c|c|c|c|}\hline
				i+1 & i+2 & \dots & i+p & 1 & \dots & i & i+2p+2 & \dots & n\\ \hline
				i+p+2 & i+p+3 & \dots & i+2p+1\\ \cline{1-4}
			\end{array}
        $$
     $$   T_{(1)}^{(i)}=\begin{tabular}{|c|}
\hline
$i+p+1$ \\
\hline
\end{tabular}$$
and the associated polynomials
		\begin{align}
        & a_{p,q}^{(i)}(x_1^+,x_2^+,x^-)=(x_1^+)^i\underbrace{\overline{x_1^+}\cdots \widetilde{x_1^+}}_{p}x^-\underbrace{\overline{x_2^+}\cdots \widetilde{x_2^+}}_{p}(x_1^+)^{q-i},\label{eq: a}\\
        & a_{p,q, w_1^+}^{(i)}(x_1^+,x_2^+,x^-)=(x_1^+)^i\underbrace{\overline{x_1^+}\cdots \widetilde{x_1^+}}_{p}x^-\underbrace{\overline{x_2^+}\cdots \widetilde{x_2^+}}_{p}(x_1^+)^{q-i}w_1^+,\label{eq: aw_1}
		\end{align}
		where the symbol $\tilde{\phantom{a}}$ or $\overline{\phantom{a}}$ means alternation on the corresponding variables. These generalized $*$-polynomials are obtained 
        from the quasi-idempotents corresponding to the pair of tableaux $\Big(T_{(p+q,p)}^{(i)},T_{(1)}^{(i)} \Big)$ by identifying all the elements in each row of $(p+q,p)$.

It is easily checked that $ a_{p,q}^{(i)}(x_1^+,x_2^+,x^-)$ and  $ a_{p,q,w_1^+}^{(i)}(x_1^+,x_2^+,x^-)$, for all $0\leq i \leq q$, are not $(W,*)$-identities of $M^D$. We will prove that they are linearly independent modulo $\I^{W,*}(M^D)$. 
		
Suppose that
        $$
        \sum_{i=0}^q \alpha_i a_{p,q}^{(i)}(x_1^+,x_2^+,x^-)+\sum_{i=0}^q \beta_i  a_{p,q, w_1^+}^{(i)}(x_1^+,x_2^+,x^-)\equiv 0 \pmod{ \I^{W,*}(M^D)}.
        $$
        
        If we evaluate $x_1^+=k(e_{11}+e_{44})+e_{22}+e_{33}$,  $x_2^+=e_{11}+e_{44}$ and $x^-=e_{12}-e_{34}$, with $k\in F$, $k \neq 0$, we get
$$
  (-1)^p \left(\sum_{i=0}^q (\alpha_i-\beta_i) k^i \right) e_{12}
        - \left( \sum_{i=0}^q (\alpha_i + \beta_i) k^{q-i} \right) e_{34} =0.
$$
Hence, setting $t=q-i$, it follows that
$$
\sum_{i=0}^q (\alpha_i-\beta_i) k^i=0, \quad \quad  \sum_{t=0}^q (\alpha_{q-t} + \beta_{q-t}) k^{t}=0.
$$
 Now, for $0\leq i \leq q$, let $\gamma_i^{(1)}:=\alpha_i-\beta_i$ and $\gamma_i^{(2)}:=\alpha_{q-i}+\beta_{q-i}$. Given $j\in \{1,2\}$, assume that there exists $\gamma_i^{(j)}\neq 0$ for some $0\leq i \leq q $.  Since $F$ is an infinite field, we can choose $k_0,\dots,k_{q}\in F$ such that $k_i\neq 0$ and $k_i\neq k_j$ for $0\leq i \neq j \leq q$. Then, we get the following homogeneous linear system of $q+1$ equations in the $q+1$ variables $\gamma_0^{(j)}, \ldots, \gamma_q^{(j)}$
     $$
        \sum_{i=0}^{q}k_s^i \gamma_i^{(j)}=0 \quad \quad 0 \leq s \leq q.
    $$
Since the matrix associated to the above system is a Vandermonde matrix, it follows that $\gamma_i^{(j)}=0$ for any $0 \leq i \leq q$, a contradiction. Hence, for all $0\leq i \leq q$,  $\alpha_i-\beta_i=\alpha_i+\beta_i=0$ and so $\alpha_i=\beta_i=0$.

Thus the $2(q+1)$ generalized $*$-polynomials $a_{p,q}^{(i)}(x_1^+,x_2^+,x^-)$, $ a_{p,q, w_1^+}^{(i)}(x_1^+,x_2^+,x^-)$, $0\leq i \leq q$, are linearly independent modulo $\I^{W,*}(M^D)$.

		Notice that for every $0\leq i \leq q$, the complete linearization of $a_{p,q}^{(i)}(x_1^+,x_2^+,x^-)$ and $ a_{p,q, w_1^+}^{(i)}(x_1^+,x_2^+,x^-)$ are, respectively, the $(W,*)$-polynomials
        $$e_{T_{(p+q,p)}^{(i)}}e_{T_{(1)}^{(i)}}(x_1^+ \cdots x_{i+p}^+ x^-_{i+p+1} x_{i+p+2}^+ \cdots x_{n}^+ ) \quad \text{and} \quad e_{T_{(p+q,p)}^{(i)}}e_{T_{(1)}^{(i)}}(x_1^+ \cdots x_{i+p}^+ x^-_{i+p+1} x_{i+p+2}^+ \cdots x_{n}^+ w_1^+),$$ 
        which are consequently linearly independent modulo $\Id^{W,*}(M^D)$.
      Hence,  it follows that the multilinear $(W,*)$-polynomials $$e_{(p+q,p),(1)}^{(i)}(x_1^+,\ldots,x_{n-1}^+,x^-_n)\quad \text{and} \quad  e_{(p+q,p),(1)}^{(i),w_1^+}(x_1^+,\ldots,x_{n-1}^+,x^-_n),$$ 
      obtained, respectively, from the expressions above by renaming the variables, are also linearly independent modulo $\Id^{W,*}(M^D)$, for $0\leq i \leq q$, and they generate $2(q+1)$ distinct copies of the same irreducible module associated to the pair of partitions $\big((p+q,p),(1)\big)$ in the module decomposition of $P_{n-1,1}(M^D)$. So, we have that $m_{(p+q,p),(1)}\geq 2(q+1)$, as claimed.

\vspace{1ex}
        
\noindent\textbf{Case 2:}  We prove that $m_{(n-1),(1)}\geq 2n$.

For $0\leq i\leq n-1$ given the tableaux
$$
T_{(n-1)}^{(i)}=\begin{array}{|c|c|c|c|c|c|}\hline
1   & \dots & i & i+2 & \dots & n\\  
\hline
\end{array}
$$
and
$$
T_{(1)}^{(i)}=\begin{array}{|c|}\hline
i+1\\  
\hline
\end{array}\,,
$$
the associated polynomials are respectively
\begin{align}
    b^{(i)}(x^+,x^-)&=(x^+)^ix^-(x^+)^{n-i-1} ,\label{eq: b_i} \\
      b_{w_1^+}^{(i)}(x^+,x^-)&=(x^+)^ix^-(x^+)^{n-i-1}w_1^+.\label{eq: b_{w_1i}}
\end{align}
These polynomials are obtained from the quasi-idempotents corresponding to the pair of tableaux $\left(T_{(n-1)}, T_{(1)}\right)$ by identifying all the elements in $T_{(n-1)}$.

 We will prove that the $2n$ $(W,*)$-polynomials $b^{(i)}(x^+,x^-)$, $ b_{w_1^+}^{(i)}(x^+,x^-)$, $0\leq i \leq n-1$, which are not $(W,*)$-identities of $M^D$, are linearly independent modulo $\Id^{W,*}(M^D)$. 
 
 Let
$$
 \sum_{0 \leq i \leq n-1} \alpha_i b^{(i)}(x^+,x^-)+ \sum_{0 \leq i \leq n-1} \beta_i b_{w_1^+}^{(i)} (x^+,x^-)\equiv 0 \pmod{ \I^{W,*}(M^D)}.
$$

By  evaluating  $x^+=k(e_{11}+e_{44})+(e_{22}+e_{33})$, $x^-=e_{12}-e_{34}$, where $k \in F$, $k\neq 0$ we get 
$$
     \left(\sum_{i=0}^{n-1}(\alpha_i-\beta_i ) k^i\right)e_{12}-\left(\sum_{i=0}^{n-1}(\alpha_i+\beta_i)k^{n-i-1}\right)e_{34}=0.
      $$
Hence, setting $t=n-i-1$, it follows
$$
\sum_{i=0}^{n-1}(\alpha_i-\beta_i ) k^i=0, \quad \quad \sum_{t=0}^{n-1}(\alpha_{n-1-t}+\beta_{n-1-t})k^{t}=0.
$$
Now, from the same argument of the previous case, we obtain $\alpha_i=\beta_i=0$ for all $0\leq i \leq n-1$. Consequently, the $(W,*)$-polynomials $b^{(i)}(x^+,x^-)$ and $b_{w_1^+}^{(i)}(x^+,x^-)$, for $0\leq i \leq n-1$, are linearly independent modulo $\I^{W,*}(M^D)$. Applying a similar reasoning as in Case 1, this implies that $m_{(n-1),(1)} \ge 2n$, as claimed.
\end{proof}

\begin{lemma}\label{lemma: cn,0}
 If $r=n$, then in \eqref{eq: ch M^D} we have  
$$
m_{\lambda,\mu}= \begin{cases}
             4(q+1) & \text{if }\lambda=(p+q,p), \mu=\emptyset\\
             2(q+1) & \text{if } \lambda=(p+q,p,1), \mu=\emptyset\\
              2n+2 & \text{if } \lambda=(n),  \mu=\emptyset\\
              0 & \text{otherwise}\\
    \end{cases}.
$$   
\end{lemma}
\begin{proof}
We proceed as in the proof of the previous lemma. Observe that, in this case, as a consequence of formula \eqref{eq: codim and ch M^D}, it is enough to prove that:
\begin{enumerate}
\item $m_{(p+q,p),\emptyset}\geq 4(q+1)$ for all $p \geq 1$ and $q \geq 0$ such that $2p+q=n$;
\vspace{1mm}
\item $m_{(p+q,p,1),\emptyset}\geq 2(q+1)$ for all $p \geq 1$ and $q \geq 0$ such that $2p+q=n-1$;
\vspace{1mm}
\item $m_{(n),\emptyset}\geq 2n+2$.
\end{enumerate}
Indeed, if these conditions hold, then  
\begin{equation}\label{eq: c_{n,0}(M^D) diseq}
		  c_{n,0}(M^D) 
          \geq\sum_{\substack{1\leq p \leq \lfloor \frac{n}{2}\rfloor\\0 \leq q \leq n-2p}}4(q+1)d_{(p+q,p)}d_{\emptyset}+\sum_{\substack{1\leq p \leq \lfloor \frac{n-1}{2}\rfloor\\0 \leq q \leq n-2p-1}}2(q+1)d_{(p+q,p,1)}d_{\emptyset}+(2n+2) d_{(n)}d_{\emptyset}. 
	\end{equation}
Notice that by the formulas $(4.8)$ and $(4.9)$ proved in Theorem 4.4 of \cite{MarRiz2}, we have
$$
  \sum_{ \substack{ 1 \leq p \leq \lfloor\frac{n}{2}\rfloor \\ 0 \leq q \leq n-2p}}  4(q+1) d_{(p+q,p)}=2^{n+2}-4n-4.
$$
and
$$
\sum_{ \substack{ 1 \leq p \leq \lfloor\frac{n-1}{2}\rfloor \\ 0 \leq q \leq n-2p-1}}  2(q+1) d_{(p+q,p,1)} =  (n-4)2^{n}+2n+4.
$$
Hence, recalling that $d_{(n)}=d_{\emptyset}=1$, we obtain
\begin{equation*}
    \sum_{\substack{1\leq p \leq \lfloor \frac{n}{2}\rfloor\\0 \leq q \leq n-2p}}4(q+1)d_{(p+q,p)}d_{\emptyset}+\sum_{\substack{1\leq p \leq \lfloor \frac{n-1}{2}\rfloor\\0 \leq q \leq n-2p-1}}2(q+1)d_{(p+q,p,1)}d_{\emptyset}+(2n+2) d_{(n)}d_{\emptyset}
    =n 2^n +2.
\end{equation*}
Since, by \eqref{eq: c_{n,0}(M^D)}, we have $c_{n,0}(M^D)=n2^n+2$, it follows that \eqref{eq: c_{n,0}(M^D) diseq} is in fact an equality. Thus it remains to prove the three inequalities above.

\vspace{1ex}
We now compute the three cases separately.

\vspace{1ex}

     \noindent\textbf{Case 1:} We prove that $m_{(p+q,p), \emptyset}\geq 4(q+1)$ for all $p \geq 1$ and $q \geq 0$ such that $2p+q=n$.

		For $0\leq i \leq q$, we consider the tableau
		\begin{equation*}
			T_{(p+q,p)}^{(i)}=\begin{array}{|c|c|c|c|c|c|c|c|c|c|c|}
				\hline
				i+1 & i+2 & \dots & i+p-1 & i+p & 1 & \dots & i & i+2p+1 & \dots & n\\
				\hline
				i+p+2 & i+p+3 & \dots & i+2p & i+p+1\\   
				\cline{1-5}
			\end{array}
		\end{equation*}
		and its associated $(W,*)$-polynomials
        
    	\begin{align}
	    f_{p,q}^{(i)}(x_1^+,x_2^+)&=(x_1^+)^i\underbrace{\overline{x_1^+}\cdots\widetilde{x_1^+}}_{p-1}[x_1^+,x_2^+]\underbrace{\overline{x_2^+}\cdots \widetilde{x_2^+}}_{p-1}(x_1^+)^{q-i},\label{eq: f}\\
		f_{p,q,w_1^+}^{(i)}(x_1^+,x_2^+)&=(x_1^+)^i\underbrace{\overline{x_1^+}\cdots\widetilde{x_1^+}}_{p-1}[x_1^+,x_2^+]\underbrace{\overline{x_2^+}\cdots \widetilde{x_2^+}}_{p-1}(x_1^+)^{q-i}w_1^+,\label{eq: fw_1}\\
		g_{p,q}^{(i)}(x_1^+,x_2^+)&=(x_1^+)^i\underbrace{\overline{x_1^+}\cdots\widetilde{x_1^+}}_{p-1}(x_1^+w_1^+ x_2^+-x_2^+w_1^+x_1^+)\underbrace{\overline{x_2^+}\cdots \widetilde{x_2^+}}_{p-1}(x_1^+)^{q-i},\label{eq: gi}\\
		g_{p,q,w_1^+}^{(i)}(x_1^+,x_2^+)&=(x_1^+)^i\underbrace{\overline{x_1^+}\cdots\widetilde{x_1^+}}_{p-1}(x_1^+w_1^+ x_2^+-x_2^+w_1^+x_1^+)\underbrace{\overline{x_2^+}\cdots \widetilde{x_2^+}}_{p-1}(x_1^+)^{q-i}w_1^+,\label{eq: gw_1}
	\end{align}
		where the symbol $\tilde{\phantom{a}}$ or $\overline{\phantom{a}}$ means alternation on the corresponding variables.  For any $0\leq i \leq q$, these $(W,*)$-polynomials are obtained 
        from the quasi-idempotents corresponding to the tableau $T_{(p+q,p)}^{(i)}$ by identifying all the elements in each row. 
        
        Clearly, the $4(q+1)$ polynomials $f_{p,q}^{(i)}(x_1^+,x_2^+)$, $f_{p,q,w_1^+}^{(i)}(x_1^+,x_2^+)$, $g_{p,q}^{(i)}(x_1^+,x_2^+)$, $g_{p,q,w_1^+}^{(i)}(x_1^+,x_2^+)$, for all $0\leq i \leq q$, do not belong to $\I^{W,*}(M^D)$. 
        Next, we will prove that they are linearly independent modulo $\I^{W,*}(M^D)$.
		
		Suppose that
        $$
        \sum_{i=0}^q \alpha_i f_{p,q}^{(i)}(x_1^+,x_2^+)+\sum_{i=0}^q\beta_i f_{p,q,w_1^+}^{(i)}(x_1^+,x_2^+)+ \sum_{i=0}^q\gamma_i g_{p,q}^{(i)}(x_1^+,x_2^+)+ \sum_{i=0}^q\delta_i g_{p,q,w_1^+}^{(i)}(x_1^+,x_2^+)\equiv 0 \pmod{\I^{W,*}(M^D)}.
        $$
        
        If we evaluate $x_1^+=k(e_{11}+e_{44})+e_{22}+e_{33}+e_{12}+e_{34}$, with $k \in F$, $k \neq 0$, and $x_2^+=e_{11}+e_{44}$, we get 
         \begin{equation}\label{eq: eval ch p+q,p 1}
         (-1)^{p-1} \left(\sum_{i=0}^q (-\alpha_i  + \beta_i   -\gamma_i  +\delta_i) k^i\right) e_{12}+ \left( \sum_{i=0}^q (\alpha_i+\beta_i+\gamma_i+ \delta_i)k^{q-i}\right)e_{34}=0.
        \end{equation}
        Also, by evaluating $x_1^+=e_{11}+e_{44}+k(e_{22}+e_{33})+e_{12}+e_{34}$, with $k \in F$, $k \neq 0$, and $x_2^+=e_{22}+e_{33}$, we obtain 
         \begin{equation}\label{eq: eval ch p+q,p 2}
         \left(\sum_{i=0}^q (\alpha_i-\beta_i-\gamma_i+\delta_i)k^{q-i}\right) e_{12} +(-1)^{p-1}\left( \sum_{i=0}^q(-\alpha_i-\beta_i+\gamma_i+\delta_i)k^{i}\right) e_{34}=0.
        \end{equation}
        Hence, setting $t=q-i$ and from \eqref{eq: eval ch p+q,p 1}, \eqref{eq: eval ch p+q,p 2} it follows that
        \begin{align*}
        &\sum_{i=0}^q k^i(-\alpha_i  + \beta_i   -\gamma_i  +\delta_i) =0, \quad \quad   \sum_{t=0}^q k^{t}(\alpha_{q-t}+\beta_{q-t}+\gamma_{q-t}+ \delta_{q-t})=0,\\
          & \sum_{t=0}^q k^{t} (\alpha_{q-t}-\beta_{q-t}-\gamma_{q-t}+\delta_{q-t})=0,  \quad \quad \sum_{i=0}^q k^{i}(-\alpha_i-\beta_i+\gamma_i+\delta_i)=0.
        \end{align*}

        Now, for $0\leq i \leq q$, let 
        \begin{equation}\label{eq: eta}
        \begin{split}
             &\eta_i^{(1)}:=-\alpha_i  + \beta_i   -\gamma_i  +\delta_i, \quad \quad  \eta_i^{(2)}:=\alpha_{q-i}+\beta_{q-i}+\gamma_{q-i}+ \delta_{q-i},\\
            &\eta_i^{(3)}:=\alpha_{q-i}-\beta_{q-i}-\gamma_{q-i}+\delta_{q-i}, \quad \quad \eta_i^{(4)}:=-\alpha_i-\beta_i+\gamma_i+\delta_i.
        \end{split}
        \end{equation}
        Given $1\leq j \leq 4$,  assume that there exists $\eta_i^{(j)}\neq 0$ for some $0\leq i \leq q$.  Since $F$ is an infinite field, we can choose $k_0,\dots,k_{q}\in F$ such that $k_i\neq 0$ and $k_i\neq k_j$ for $0\leq i \neq j \leq q$. Then, we get the following homogeneous linear system of $q+1$ equation in the $q+1$ variables $\eta_0^{(j)}, \ldots, \eta_q^{(j)}$
        $$
        \sum_{i=0}^q  k_s^i \eta_i^{(j)}=0, \quad \quad 0 \leq s \leq q.
        $$
        Since the matrix associated to the above system is a Vandermonde matrix, it follows that $\eta_i^{(j)}=0$ for any $0 \leq i \leq q$, a contradiction. Thus, by \eqref{eq: eta} we get that $\alpha_i=\beta_i=\gamma_i=\delta_i=0$ for all $0\leq i \leq q$. 

        Therefore, the $4(q+1)$ generalized $*$-polynomials $f_{p,q}^{(i)}(x_1^+,x_2^+)$, $f_{p,q,w_1^+}^{(i)}(x_1^+,x_2^+)$, $g_{p,q}^{(i)}(x_1^+,x_2^+)$, $g_{p,q, w_1^+}^{(i)}(x_1^+,x_2^+)$ are linearly independent modulo $\I^{W,*}(M^D)$.

Proceeding as in the previous lemma, we observe that for every $0\leq i \leq q$, the complete linearizations of $f_{p,q}^{(i)}(x_1^+,x_2^+)$, $f_{p,q,w_1^+}^{(i)}(x_1^+,x_2^+)$, $g_{p,q}^{(i)}(x_1^+,x_2^+)$, and $g_{p,q, w_1^+}^{(i)}(x_1^+,x_2^+)$ are, respectively, the $(W,*)$-polynomials 
        $$e_{T_{(p+q,p)}^{(i)}}(x_1^+ \cdots x_{n}^+), \quad e_{T_{(p+q,p)}^{(i)}}(x_1^+ \cdots x_{n}^+ w_1^+),
        $$
        $$
        e_{T_{(p+q,p)}^{(i)}}(x_1^+ \cdots x_{i+p}^+ w_1^+ x_{i+p+1}^+ \cdots  x_{n}^+), \quad e_{T_{(p+q,p)}^{(i)}}(x_1^+ \cdots x_{i+p}^+ w_1^+ x_{i+p+1}^+ \cdots  x_{n}^+  w_1^+).$$
         Hence, these multilinear $(W,*)$-polynomials are also linearly independent modulo $\Id^{W,*}(M^D)$. Consequently, they generate $4(q+1)$ distinct copies of the same irreducible module corresponding to the pair of partitions $\big((p+q,p),\emptyset\big)$ in the module decomposition of $P_{n,0}(M^D)$. Therefore, we conclude that $m_{(p+q,p),\emptyset}\geq 4(q+1)$, as claimed.

		\vspace{1ex}
		
		\noindent\textbf{Case 2:} We prove that $m_{(p+q,p,1),\emptyset}\geq 2(q+1)$ for all $p \geq 1$ and $q \geq 0$ such that $2p+q=n-1$.
		
		For $0\leq i \leq q$, consider the tableau
		\begin{equation*}
			T_{(p+q,p,1)}^{(i)}=\begin{array}{|c|c|c|c|c|c|c|c|c|c|c|}
				\hline
				i+p & i+1 & \dots & i+p-1 & 1 & \dots & i & i+2p+2 & \dots & n\\
				\hline
				i+p+1 & i+p+3 & \dots &  i+2p+1\\   
				\cline{1-4}
				i+p+2\\
				\cline{1-1}
			\end{array}
		\end{equation*}
		and its associated $(W,*)$-polynomials
		\begin{align}
		   h_{p,q}^{(i)}(x_1^+,x_2^+,x_3^+)&=(x_1^+)^i\underbrace{\widehat{x_1^+}\cdots \widetilde{x_1^+}}_{p-1}\overline{x_1^+}\,\overline{x_2^+} \,\overline{x_3^+}\underbrace{\widehat{x_2^+}\cdots \widetilde{x_2^+}}_{p-1}(x_1^+)^{q-i},\label{eq: h}\\
		h_{p,q,w_1^+}^{(i)}(x_1^+,x_2^+,x_3^+)&=(x_1^+)^i\underbrace{\widehat{x_1^+}\cdots \widetilde{x_1^+}}_{p-1}\overline{x_1^+}\,\overline{x_2^+} \,\overline{x_3^+}\underbrace{\widehat{x_2^+}\cdots \widetilde{x_2^+}}_{p-1}(x_1^+)^{q-i}w_1^+, \label{eq: hw_1}
		\end{align}
		 where the symbol $\hat{}$ or $\tilde{}$ or $\bar{}$ means alternation on the corresponding variables. For any $0\leq i \leq q$, each of these generalized $*$-polynomials is obtained from the quasi-idempotents corresponding to the tableau $T_{(p+q,p,1)}^{(i)}$ by identifying all the elements in each row.

         Clearly, $h^{(i)}_{p,q}(x_1^+,x_2^+,x_3^+)$,  $h_{p,q,w_1^+}^{(i)}(x_1^+,x_2^+,x_3^+)$, for $0\leq i \leq q$,  do not belong to $\I^{W,*}(M^D)$ and we will prove that they are linearly independent modulo $\I^{W,*}(M^D)$.

         We proceed as in the previous case. So, assume that
         $$
        \sum_{i=0}^q \alpha_i h_{p,q}^{(i)}(x_1^+,x_2^+,x_3^+)+
        \sum_{i=0}^q \beta_i h_{p,q,w_1^+}^{(i)}(x_1^+,x_2^+,x_3^+)\equiv 0 \pmod{\I^{W,*}(M^D)}.
        $$
        With the same reasoning as before, if we evaluate first $x_1^+=k(e_{11}+e_{44})+e_{22}+e_{33}+e_{12}+e_{34}$, $x_2^+=e_{11}+e_{44}$, $x_3^+=e_{22}+e_{33}$ with $0\neq k \in F$, we get
        $$
        (-1)^{p} \left(\sum_{i=0}^q (\alpha_i-\beta_i)k^i  \right) e_{12}+ \left(\sum_{i=0}^q (\alpha_i+\beta_i)k^{q-i}  \right) e_{34} =0,
        $$
       and, as a consequence,
       $$
       \sum_{i=0}^q (\alpha_i-\beta_i)k^i =0, \quad \quad \sum_{i=0}^q (\alpha_i+\beta_i)k^{q-i}=0.
       $$
Now, by applying the same argument as in Case $1$ of the proof of Lemma \ref{lemma: cn-1,1}, it follows that $\alpha_i=\beta_i=0$ for all $0\leq i \leq q$. Hence the $2(q+1)$ generalized $*$-polynomials $h_{p,q}^{(i)}(x_1^+,x_2^+,x_3^+)$ and $h_{p,q, w_1^+}^{(i)}(x_1^+,x_2^+,x_3^+)$ are linearly independent modulo $\I^{W,*}(M^D)$. 

Therefore, as in the previous case, we obtain $m_{(p+q,p,1), \emptyset}\geq 2(q+1)$, as claimed.

\vspace{1ex}
        
		\noindent \textbf{Case 3:} We prove that $m_{(n),\emptyset}\geq 2n+2$.

        Given the tableau
        \begin{equation*}
			T_{(n)}=\begin{array}{|c|c|c|c|c|c|c|c|c|c|c|}
				\hline
				1 & 2 & 3 & \dots & n\\  
				\hline
			\end{array}
		\end{equation*}
		and its associated polynomials
        \begin{align}
            t(x^+)&=(x^+)^n, \label{eq: t}\\
            t_{w_1^+}(x^+)&=(x^+)^nw_1^+,\label{eq: tw1}\\
            u^{(i)}(x^+)&=(x^+)^i[x^+,w_1^+](x^+)^{n-i-1}, \quad \quad  0\leq i\leq n-1, \label{eq: u}\\
            u_{w_1^+}^{(i)}(x^+)&=(x^+)^i[x^+,w_1^+](x^+)^{n-i-1}w_1^+,  \quad \quad  0\leq i\leq n-1, \label{eq: uw1}
        \end{align}
        These $(W,*)$-polynomials are obtained 
        from the quasi-idempotents corresponding to the tableau $T_{(n)}$ by identifying all the elements.

		We will prove that the $2n+2$ generalized $*$-polynomials $ t(x^+)$, $ t_{w_1^+}(x^+)$, $u^{(i)}(x^+)$,  and $u_{w_1^+}^{(i)}(x^+)$, for $0\leq i \leq n-1$, which are not $(W,*)$-identities of $M^D$, are linearly independent modulo $\Id^{W,*}(M^D)$. 
        
        Suppose that
        \begin{align*}
        \alpha t(x^+)+ \beta t_{w_1^+} (x^+)+\sum_{0 \leq i \leq n-1}\gamma_i u^{(i)}(x^+)+\sum_{0 \leq i \leq n-1} \delta_i u_{w_1^+}^{(i)} (x^+)\equiv 0 \pmod{\I^{W,*}(M^D)}.
        \end{align*}

        First of all, note that if we evaluate them at $x^+=e_{11}+e_{44}+e_{22}+e_{33}$, then the last two sums vanish and we get
		$(\alpha+\beta)(e_{11}+e_{44})+(\alpha-\beta)(e_{22}+e_{33})= 0,$
		so $\alpha=\beta=0$.
        
		Now, we want to prove that the last two sums are also linearly independent.
        By applying the same arguments as in the previous cases and evaluating them at $x^+=k(e_{11}+e_{44})+e_{22}+e_{33}+e_{12}+e_{34}$, with $0\neq k \in F$, we get 
        $$
        -2\left( \sum_{i=0}^{n-1}(\gamma_i- \delta_i)k^i\right)e_{12}+2\left( \sum_{i=0}^{n-1}(\gamma_i+ \delta_i)k^{n-1-i}\right)e_{34}=0
        $$
        and, as a consequence, we have 
        \begin{equation*}
            \sum_{i=0}^{n-1}(\delta_i-\gamma_i)k^i=0, \quad  \quad
            \sum_{i=0}^{n-1}(\gamma_i+\delta_i)k^{n-1-i}=0.
        \end{equation*}
        Again, from the same argument of Case $1$ of the proof of  Lemma \ref{lemma: cn-1,1}, we obtain that 
        $\gamma_i=\delta_i=0$ for every $0 \leq i \leq n-1$. 
        
        Thus,  the $(W,*)$-polynomials $ t(x^+)$, $ t_{w_1^+}(x^+)$, $u^{(i)}(x^+)$,  and $u_{w_1^+}^{(i)}(x^+)$, for $0\leq i \leq n-1$, are linearly independent, and as in the previous cases, we obtain $m_{(n),\emptyset}\geq 2n+2$, as desired.
\end{proof}

Now, we can finally prove the main result of this section.

\begin{theorem}\label{th: cocharacter} 
Let $M^D$ be the $(W,*)$-algebra $M$ where $W$ acts on it as the algebra $D$ by left and right multiplication. For $n\geq 1$ and $0\leq r \leq n$, consider 
 $$\chi_{r,n-r}^{W,*}(M^D)=\sum_{(\lambda,\mu)\vdash (r,n-r)}m_{\lambda,\mu}(\chi_\lambda \otimes \chi_\mu),$$
 the $(r,n-r)$-th $(W,*)$-cocharacter of $M^D$.  Then: 
 \begin{enumerate}
 \item for $r=n$,
$$
m_{\lambda,\mu}= \begin{cases}
             4(q+1) & \text{if }\lambda=(p+q,p), \mu=\emptyset\\
             2(q+1) & \text{if } \lambda=(p+q,p,1), \mu=\emptyset\\
              2n+2 & \text{if } \lambda=(n),  \mu=\emptyset\\
              0 & \text{otherwise}\\
    \end{cases};
$$
\vspace{1mm}
     \item for $r=n-1$,
$$
m_{\lambda,\mu}= \begin{cases}
             2(q+1)  & \text{if } \lambda=(p+q,p),  \mu=(1)\\
              2n & \text{if }  \lambda=(n-1),  \mu=(1)\\
            0 & \text{otherwise}
    \end{cases};
$$
\vspace{1mm}

\item for $0\leq r \leq n-2$, we have $m_{\lambda,\mu}=0$.
 \end{enumerate}
\end{theorem}
\begin{proof}
Items (1) and (2) follow from Lemmas \ref{lemma: cn,0} and \ref{lemma: cn-1,1}, respectively. Moreover, 
for $0\leq r \leq n-2$, since $c_{r,n-r}(M^D)=0$, it follows from \eqref{eq: codim and ch M^D} that $m_{\lambda,\mu}=0$, which proves item (3). 
\end{proof}

\section{Growth of the generalized $*$-codimensions of $M$}

In this section, we investigate the growth of $\V^{W,*}(M^D)$ and $\V^{W,*}(M^F)$. To this end, recall that a $(W,*)$-variety $\mathcal{V}$ is said to have almost polynomial growth if $c_n^{W,*}(\mathcal{V})$ is not polynomially bounded but every proper $(W,*)$-subvariety $\mathcal{U}$ of $\mathcal{V}$ has polynomial growth. 

 If we consider $M^F$, then $\Phi(W)\cong F$. Hence, we are dealing with the ordinary $*$-identities, and consequently, by \cite[Theorem 3]{MisVal}, we obtain the following result.

\begin{theorem}\label{th: MFapg}
    $\V^{W,*}(M^F)$ has almost polynomial growth.
\end{theorem}

Next, we show that $\V^{W,*}(M^D)$ also has almost polynomial growth.
If $\mathcal{V}$ is a $(W,*)$-variety, then for $n \geq 1$ and $0\leq r\leq n$, by \eqref{eq: ch} we write
$$
\chi_{r,n-r}^{W,*}(\mathcal{V})=\sum_{(\lambda,\mu)\vdash (r,n-r)}m_{\lambda,\mu}^{\mathcal{V}}(\chi_{\lambda} \otimes \chi_{\mu})
$$
where $m_{\lambda,\mu}^{\mathcal{V}}$ denotes the multiplicity of $\chi_{\lambda}\otimes \chi_{\mu}$ in $\chi_{r,n-r}^{W,*}(\mathcal{V})$.

\begin{lemma}\label{lem: Id}
For any $\alpha,\beta \in F$ we have:
    $$([x^+,x^-](\alpha1_W+\beta w_1^+))^*\equiv [x^+,x^-](\alpha1_W-\beta w_1^+)\pmod{\I^{W,*}(M^D)},$$
    and
    $$((x^+x^-+x^-x^+)(\alpha 1_W+\beta w_1^+))^*\equiv (x^+x^-+x^-x^+)(-\alpha1_W+\beta w_1^+)\pmod{\I^{W,*}(M^D)}.$$
\end{lemma}
\begin{proof}
Since $\alpha1_W+\beta w_1^+$ is a symmetric element of $W$, $[x^+,x^-]$ is a symmetric polynomial, and $(x^+x^-+x^-x^+)$ is a skew-symmetric polynomial,  we have 
$$([x^+,x^-](\alpha1_W+\beta w_1^+))^*=(\alpha1_W+\beta w_1^+)[x^+,x^-],$$
and
$$
((x^+x^-+x^-x^+)(\alpha 1_W+\beta w_1^+))^*=- (\alpha 1_W+\beta w_1^+)(x^+x^-+x^-x^+).
$$
Then,
 since $[x^+,w_1^+]$ is skew-symmetric, $ x_1^-x_2^-, \, w_1^+x^-+x^-w_1^+\in \I^{W,*}(M^D)$ and by Lemma \ref{Lem: consequence x_1^- x_2^-}, we obtain the desired conclusion.
\end{proof}

\begin{lemma}\label{lemma: M,N-M}
   Let $\mathcal{U}$ be a proper subvariety of $\V^{W,*}(M^D)$. Then there exist constants $K\leq N$ such that
    $$
    (x_1^+)^{K}x_2^-(x_1^+)^{N-K}+\hat{\alpha} (x_1^+)^{K}x_2^-(x_1^+)^{N-K}  w_1^+ +\sum_{i<K}\hat{\beta}_i(x_1^+)^ix_2^-(x_1^+)^{N-i}+ \sum_{i<K}\hat{\gamma}_i(x_1^+)^ix_2^-(x_1^+)^{N-i}w_1^+\in \I^{W,*}(\mathcal{U}),
$$
where $\hat{\alpha}, \hat{\beta}_i, \hat{\gamma}_i \in F$.
\end{lemma}
\begin{proof} 
    Since $\mathcal{U}\subsetneq \mathcal{V}=\V^{W,*}(M^D)$,  there exist $ n \geq 1,$ $0\leq  r \leq n,$ and partitions $\lambda\vdash r$ and $\mu\vdash n-r$ such that $m_{\lambda,\mu}^{\mathcal{U}}<m_{\lambda,\mu}^{\mathcal{V}}$. Then, by Theorem \ref{th: cocharacter} we have that either
    \begin{equation}\label{eq: caso 1}
        \sum_{i=0}^q  \left(\alpha_i a_{p,q}^{(i)}(x_1^+,x_2^+,x^-)+ \beta_i a_{p,q,w_1^+}^{(i)}(x_1^+,x_2^+,x^-)\right)\in \Id^{W,*}(\mathcal{U})
        \end{equation}
        with $\alpha_i,\beta_i$ not all zero, where the polynomials $a_{p,q}^{(i)}$ and $a_{p,q,w_1^+}^{(i)}$ are defined in \eqref{eq: a} and \eqref{eq: aw_1}, respectively, in the proof of Lemma \ref{lemma: cn-1,1}, or
        \begin{equation}\label{eq: case 2}
        \sum_{i=0}^{n-1}\left(\gamma_i b^{(i)}(x^+,x^-)+ \delta_i b_{w_1^+}^{(i)}(x^+,x^-)\right)\in \Id^{W,*}(\mathcal{U})
    \end{equation}
    with $\gamma_i,\delta_i$ not all zero, where the polynomials $b^{(i)}$ and $b_{w_1^+}^{(i)}$ are defined in \eqref{eq: b_i} and \eqref{eq: b_{w_1i}}, respectively, in the proof of Lemma \ref{lemma: cn-1,1}, or
    \begin{equation}\label{eq: case 3}
        \sum_{i=0}^{q} \left(\varepsilon_i f_{p,q}^{(i)}(x_1^+,x_2^+)+ \zeta_i f_{p,q,w_1^+}^{(i)}(x_1^+,x_2^+)+ \eta_i g_{p,q}^{(i)}(x_1^+,x_2^+)+ \theta_i g_{p,q,w_1^+}^{(i)}(x_1^+,x_2^+)\right)\in \Id^{W,*}(\mathcal{U})
    \end{equation}
    with $\varepsilon_i,\zeta_i, \eta_i, \theta_i$ not all zero,  where the polynomials $f_{p,q}^{(i)}$, $f_{p,q,w_1^+}^{(i)}$, $g_{p,q}^{(i)}$ and $g_{p,q,w_1^+}^{(i)}$ are defined in \eqref{eq: f}, \eqref{eq: fw_1}, \eqref{eq: gi}, and \eqref{eq: gw_1}, respectively, in the proof of Lemma \ref{lemma: cn,0}, or
    \begin{equation}\label{eq: case 4}
        \sum_{i=0}^{q}\left(\nu_i h_{p,q}^{(i)}(x_1^+,x_2^+,x_3^+)+ \pi_i h_{p,q,w_1^+}^{(i)}(x_1^+,x_2^+,x_3^+)\right)\in \Id^{W,*}(\mathcal{U})
    \end{equation}
    with $\nu_i,\pi_i$ not all zero, where the polynomials $h_{p,q}^{(i)}$ and  $h_{p,q,w_1^+}^{(i)}$ are defined in \eqref{eq: h} and \eqref{eq: hw_1}, respectively, in the proof of Lemma \ref{lemma: cn,0}, or
     \begin{equation}\label{eq: case 5}
        \rho t(x^+)+\sigma t_{w_1^+}(x^+)+\sum_{i=0}^{n-1}\left(\tau_iu^{(i)}(x^+)+\xi_i u_{w_1^+}^{(i)}(x^+)\right)\in \Id^{W,*}(\mathcal{U})
  \end{equation}
    with $\rho,\sigma,\tau_i,\xi_i$ not all zero,  where the polynomials $t$, $t_{w_1^+}$, $u^{(i)}$, and $u_{w_1^+}^{(i)}$
    are defined in \eqref{eq: t}, \eqref{eq: tw1}, \eqref{eq: u}, and \eqref{eq: uw1}, respectively, in the proof of Lemma \ref{lemma: cn,0}.

	\vspace{1ex}
		
		\noindent\textbf{Case 1:} Suppose hold \eqref{eq: case 3}. Then
        \begin{align*}
&\sum_{i=0}^q(x_1^+)^i\underbrace{\overline{x_1^+}\cdots \widetilde{x_1^+}}_{p-1}[x_1^+,x_2^+]\underbrace{\overline{x_2^+}\cdots \widetilde{x_2^+}}_{p-1}(x_1^+)^{q-i}v_i\\
&+\sum_{i=0}^q(x_1^+)^i\underbrace{\overline{x_1^+}\cdots \widetilde{x_1^+}}_{p-1}\left(x_1^+w_1^+x_2^+-x_2^+w_1^+x_1^+\right)\underbrace{\overline{x_2^+}\cdots \widetilde{x_2^+}}_{p-1}(x_1^+)^{q-i}v'_i\in \I^{W,*}(\mathcal{U}),
\end{align*}
where $v_i := \varepsilon_i1_W+\zeta_iw_1^+\in W$ and $v'_i:= \eta_i1_W+\theta_iw_1^+\in W$  are not all zero.  If we substitute $x_2^+$ with $x_3^+ +x_4^+$, we obtain
\begin{align*}
&\sum_{i=0}^q(x_1^+)^i\underbrace{\overline{x_1^+}\cdots \widetilde{x_1^+}}_{p-1}[x_1^+,x_3^++x_4^+]\underbrace{\overline{(x_3^+ +x_4^+)}\cdots \widetilde{(x_3^+ +x_4^+)}}_{p-1}(x_1^+)^{q-i}v_i\\
&+
\sum_{i=0}^q(x_1^+)^i\underbrace{\overline{x_1^+}\cdots \widetilde{x_1^+}}_{p-1}
(x_1^+w_1^+(x_3^++x_4^+)-(x_3^++x_4^+)w_1^+x_1^+) \underbrace{\overline{(x_3^++x_4^+)}\cdots \widetilde{(x_3^++x_4^+)}}_{p-1}(x_1^+)^{q-i}v'_i \in \I^{W,*}(\mathcal{U}).
\end{align*}
Now, consider the homogeneous component $c$ of degree $1$ in $x_4^+$. 

Since the $(W,*)$-polynomial  $[x_2^-,w_1^+]$ is symmetric, we may substitute $x_3^+$ by $(x_1^+)^2$ and $x_4^+$ by $[x_2^-,w_1^+]$ in $c$. Therefore, since the $(W,*)$-polynomial $x_1^+w_1^+(x_1^+)^2-(x_1^+)^2w_1^+x_1^+$ is skew-symmetric and $w_1^+x^-+x^-w_1^+,\, x_1^-x_2^-\in \Id^{W,*}(M^D)\subset \Id^{W,*}(\mathcal{U})$, by Lemma \ref{Lem: consequence x_1^- x_2^-} it follows that 
$$
\left(x_1^+w_1^+(x_1^+)^2-(x_1^+)^2w_1^+x_1^+\right) (x_1^+)^k [x_2^-,w_1^+], \  [x_2^-,w_1^+] (x_1^+)^k\left(x_1^+w_1^+(x_1^+)^2-(x_1^+)^2w_1^+x_1^+\right) \ \in  \Id^{W,*}(\mathcal{U})
$$
for all $k\geq 0$, and as a consequence, we have
\begin{align*}
&\sum_{i=0}^q(x_1^+)^i\underbrace{\overline{x_1^+}\cdots \widetilde{x_1^+}}_{p-1}\left( x_1^+[x_2^-,w_1^+]- [x_2^-,w_1^+]x_1^+ \right)\underbrace{\overline{(x_1^+)^2}\cdots \widetilde{(x_1^+)^2}}_{p-1}(x_1^+)^{q-i}v_i\\
&
+ \sum_{i=0}^q(x_1^+)^i\underbrace{\overline{x_1^+}\cdots \widetilde{x_1^+}}_{p-1}
(x_1^+w_1^+[x_2^-,w_1^+]-[x_2^-,w_1^+]w_1^+ x_1^+) \underbrace{\overline{(x_1^+)^2}\cdots \widetilde{(x_1^+)^2}}_{p-1}(x_1^+)^{q-i}v'_i \in \I^{W,*}(\mathcal{U}).
\end{align*}
Since $[x_2^-,w_1^+]\equiv 2 x_2^-w_1^+ \pmod{\I^{W,*}(M^D)}$ and $\I^{W,*}(M^D)\subset \I^{W,*}(\mathcal{U})$, also $[x_1^+,w_1^+]$ is skew-symmetric, $ x_1^-x_2^-\in \I^{W,*}(M^D) \subset \I^{W,*}(\mathcal{U})$ and by Lemma \ref{Lem: consequence x_1^- x_2^-}, we obtain
\begin{align*}
&\sum_{i=0}^q(x_1^+)^i\underbrace{\overline{x_1^+}\cdots \widetilde{x_1^+}}_{p-1}\left( x_1^+x_2^-- x_2^-x_1^+ \right)\underbrace{\overline{(x_1^+)^2}\cdots \widetilde{(x_1^+)^2}}_{p-1}(x_1^+)^{q-i}v''_i\\
&
+ \sum_{i=0}^q(x_1^+)^i\underbrace{\overline{x_1^+}\cdots \widetilde{x_1^+}}_{p-1}
(x_1^+w_1^+x_2^-w_1^+ -x_2^-(w_1^+)^2 x_1^+) \underbrace{\overline{(x_1^+)^2}\cdots \widetilde{(x_1^+)^2}}_{p-1}(x_1^+)^{q-i} v'''_i \in \I^{W,*}(\mathcal{U}),
\end{align*}
where $v''_i:= 2\zeta_i1_W+ 2\varepsilon_iw_1^+\in W$ since $x^+ (w_1^+)^2- x^+ \in \I^{W,*}(M^D) \subset \I^{W,*}(\mathcal{U})$, $v'''_i:=2\eta_i1_W+2\theta_iw_1^+\in W$.
Now, since $(w_1^+)^2x^{+}-x^{+}, \, x^{+}(w_1^+)^2-x^{+}, \, w_1^+x_2^- +x_2^- w_1^+ \in \I^{W,*}(M^D) \subset \I^{W,*}(\mathcal{U})$, we get
$$
x_1^+w_1^+x_2^- w_1^+ -x_2^- (w_1^+)^2x_1^+\equiv -x_1^+x_2^- -x_2^- x_1^+ \pmod{\I^{W,*}(\mathcal{U})},
$$
and so it follows that
\begin{equation*}\label{eq: c(x_1,x_2)'}
\begin{split}
&\sum_{i=0}^q(x_1^+)^i\underbrace{\overline{x_1^+}\cdots \widetilde{x_1^+}}_{p-1}[x_1^+,x_2^-]\underbrace{\overline{(x_1^+)^2}\cdots \widetilde{(x_1^+)^2}}_{p-1}(x_1^+)^{q-i}v''_i\\
&
- \sum_{i=0}^q(x_1^+)^i\underbrace{\overline{x_1^+}\cdots \widetilde{x_1^+}}_{p-1}
(x_1^+ x_2^- + x_2^- x_1^+)\underbrace{\overline{(x_1^+)^2}\cdots \widetilde{(x_1^+)^2}}_{p-1}(x_1^+)^{q-i}v'''_i \ \in \I^{W,*}(\mathcal{U}),
\end{split}
\end{equation*}
 where $v''_i= 2\zeta_i1_W+ 2\varepsilon_iw_1^+\in W$ and $v'''_i=2\eta_i1_W+2\theta_iw_1^+\in W$ are not all zero.

Let $t$ be the maximum $i$ such that $v''_i\neq 0$ or $v'''_i\neq 0$.  Since $[x_1^+,w_1^+]$ is skew-symmetric and $x_1^-x_2^-\in \I^{W,*}(M^D)\subset \I^{W,*}(\mathcal{U}) $, we can move $v''_i, \, v'''_i$ and obtain 
\begin{equation}\label{eq: c(x_1,x_2)}
\begin{split}
&(x_1^+)^t\underbrace{\overline{x_1^+}\cdots \widetilde{x_1^+}}_{p-1}\left( [x_1^+,x_2^-]v''_t- (x_1^+ x_2^- + x_2^- x_1^+) v'''_t\right)\underbrace{\overline{(x_1^+)^2}\cdots \widetilde{(x_1^+)^2}}_{p-1}(x_1^+)^{q-t}\\
&+\sum_{i<t}(x_1^+)^i\underbrace{\overline{x_1^+}\cdots \widetilde{x_1^+}}_{p-1}\left( [x_1^+,x_2^-]v''_i- (x_1^+ x_2^- + x_2^- x_1^+) v'''_i\right)\underbrace{\overline{(x_1^+)^2}\cdots \widetilde{(x_1^+)^2}}_{p-1}(x_1^+)^{q-i} \ \in \I^{W,*}(\mathcal{U}).
\end{split}
\end{equation}
First, observe that we may assume, without loss of generality, that $t=q-t$. Indeed, if $t<q-t$, we multiply \eqref{eq: c(x_1,x_2)} on the left by $(x_1^+)^{q-2t}$, 
thereby reducing to the case $t=q-t$.
Similarly, if $t>q-t$, we can multiply \eqref{eq: c(x_1,x_2)} on the right by $(x_1^+)^{2t-q}$. Hence, \eqref{eq: c(x_1,x_2)} becomes
\begin{equation*}
\begin{split}
r(x_1^+,x_2^-):=&(x_1^+)^t\underbrace{\overline{x_1^+}\cdots \widetilde{x_1^+}}_{p-1}\left( [x_1^+,x_2^-]v''_t- (x_1^+ x_2^- + x_2^- x_1^+) v'''_t\right)\underbrace{\overline{(x_1^+)^2}\cdots \widetilde{(x_1^+)^2}}_{p-1}(x_1^+)^{t}\\
&+\sum_{i<t}(x_1^+)^i\underbrace{\overline{x_1^+}\cdots \widetilde{x_1^+}}_{p-1}\left( [x_1^+,x_2^-]v''_i- (x_1^+ x_2^- + x_2^- x_1^+) v'''_i\right)\underbrace{\overline{(x_1^+)^2}\cdots \widetilde{(x_1^+)^2}}_{p-1}(x_1^+)^{2t-i} \ \in \I^{W,*}(\mathcal{U}),
\end{split}
\end{equation*}
with $v''_j= 2\zeta_j1_W+ 2\varepsilon_jw_1^+\in W$ and $v'''_j=2\eta_j1_W+2\theta_jw_1^+\in W$.
Since $(r(x_1^+,x_2^-))^*\in \I^{W,*}(\mathcal{U})$, by Lemma \ref{lem: Id} we have
\begin{equation*}
\begin{split}
r'(x_1^+,x_2^-):=&(x_1^+)^t\underbrace{\overline{x_1^+}\cdots \widetilde{x_1^+}}_{p-1}\left( [x_1^+,x_2^-]u_t- (x_1^+ x_2^- + x_2^- x_1^+) u'_t\right)\underbrace{\overline{(x_1^+)^2}\cdots \widetilde{(x_1^+)^2}}_{p-1}(x_1^+)^{t}\\
&+\sum_{i<t}(x_1^+)^i\underbrace{\overline{x_1^+}\cdots \widetilde{x_1^+}}_{p-1}\left( [x_1^+,x_2^-]u_i- (x_1^+ x_2^- + x_2^- x_1^+) u'_i\right)\underbrace{\overline{(x_1^+)^2}\cdots \widetilde{(x_1^+)^2}}_{p-1}(x_1^+)^{2t-i} \ \in \I^{W,*}(\mathcal{U}),
\end{split}
\end{equation*}
 where $u_t= 2\zeta_t1_W- 2\varepsilon_tw_1^+\in W$, $u'_t=-2\eta_t1_W+2\theta_tw_1^+\in W$, and for $i<t$, $u_i= 2\zeta_{2t-i}1_W- 2\varepsilon_{2t-i}w_1^+\in W$ and $u'_i=-2\eta_{2t-i}1_W+2\theta_{2t-i}w_1^+\in W$. 
Moreover, notice that 
$$v''_t+u_t=4\zeta_t1_W, \quad v'''_t+u'_t=4\theta_tw_1^+, \quad v''_t-u_t=4\varepsilon_tw_1^+, \quad    v'''_t-u'_t=4\eta_t1_W.$$

Suppose first that $\zeta_t\neq 0$ or $\theta_t\neq 0$ and  set $N=q+3p-2.$ Since  $r(x_1^+,x_2^-)+ r'(x_1^+,x_2^-)\in \I^{W,*}(\mathcal{U}),$ we have 
\begin{equation*}
\begin{split}
\zeta_t(x_1^+)^{t+2p-1} x_2^- (x_1^+)^{N-t-2p+1}  &-\theta_t(x_1^+)^{t+2p-1} x_2^- (x_1^+)^{N-t-2p+1} w_1^+\\ &+\sum_{i<t+2p-1} \beta'_i(x_1^+)^{i} x_2^- (x_1^+)^{N-i} +\sum_{i<t+2p-1} \gamma'_i(x_1^+)^{i} x_2^- (x_1^+)^{N-i} w_1^+ \ \in \I^{W,*}(\mathcal{U}),
\end{split}
\end{equation*}
for some $\beta'_i, \gamma'_i\in F$. Since  $\zeta_t\neq 0$ or $\theta_t\neq 0$ and and $x^{\epsilon}(w_1^+)^2-x^{\epsilon} \in  \I^{W,*}(M^D)\subset  \I^{W,*}(\mathcal{U})$, $\epsilon\in \{+,-\}$, we have 
 \begin{align*}
    (x_1^+)^{t_1+2p-1}x_2^- (x_1^+)^{N-t_1-2p+1} &+ \hat{\alpha} (x_1^+)^{t_1+2p-1}x_2^- (x_1^+)^{N-t_1-2p+1} w_1^+ \\
    &+\sum_{i<t_1+2p-1}\hat{\beta}_i(x_1^+)^ix_2^-(x_1^+)^{N-i} +\sum_{i<t_1+2p-1}\hat{\gamma}_i(x_1^+)^ix_2^-(x_1^+)^{N-i}w_1^+ \ \in \I^{W,*}(\mathcal{U}),
\end{align*}
for some $\hat{\alpha}, \hat{\beta}_i, \hat{\gamma}_i \in F$. Hence, setting $K:=t+2p-1$ gives the desired conclusion.

If $\zeta_t=\theta_t = 0$, then $\varepsilon_t\neq 0$ or $\eta_t\neq 0$. Thus, since $r(x_1^+,x_2^-)- r'(x_1^+,x_2^-)\in \I^{W,*}(\mathcal{U}),$ by proceeding as above we obtain the desired conclusion.

\vspace{1ex}

\noindent\textbf{Case 2:} Suppose hold \eqref{eq: case 5}. Then
\begin{equation}\label{eq: t,u}
    (x^+)^nv+\sum_{i=0}^{n-1}(x^+)^i[x^+,w_1^+](x^+)^{n-i-1}v'_i\in \I^{W,*}(\mathcal{U}),
\end{equation}
where $v:=\rho1_W+\sigma w_1^+\in W$ and $v'_i:=\tau_i 1_W+\xi_iw_1^+\in W$ are not all zero.

Let us substitute $x^+$ with $x_1^+ +x_2^+$ in \eqref{eq: t,u}, and consider the homogeneous component $c$ of degree $1$ in $x_2^+$. 
Since the $(W,*)$-polynomial  $[x_2^-,w_1^+]$ is symmetric, we may substitute $x_2^+$ by $[x_2^-,w_1^+]$ in $c$. Therefore, since the $(W,*)$-polynomial $[x_2^+,w_1^+]$ is skew-symmetric and $w_1^+x^-+x^-w_1^+,\, x_1^-x_2^-\in \Id^{W,*}(M^D)\subset \Id^{W,*}(\mathcal{U})$, by Lemma \ref{Lem: consequence x_1^- x_2^-} it follows that 
$
[x_2^+,w_1^+] (x_1^+)^k [x_2^-,w_1^+], \,  [x_2^-,w_1^+] (x_1^+)^k[x_2^+,w_1^+] \in  \Id^{W,*}(\mathcal{U})
$
for all $k\geq 0$, and as a consequence, we have
$$
\sum_{i=0}^{n-1}
(x_1^+)^i [x_2^-,w_1^+] (x_1^+)^{n-i-1} v+\sum_{i=0}^{n-1}(x_1^+)^i[[x_2^-,w_1^+],w_1^+](x_1^+)^{n-i-1}v'_i\in \I^{W,*}(\mathcal{U}).
$$
Now, since $ [x_2^-,w_1^+]-2 x_2^-w_1^+, \, x_2^-(w_1^+)^2-x_2^-\in \I^{W,*}(M^D) \subset \I^{W,*}(\mathcal{U})$ and since
$[x_1^+,w_1^+]$ is a skew-symmetric element, by Lemma \ref{Lem: consequence x_1^- x_2^-} we have
$$
\sum_{i=0}^{n-1}
(x_1^+)^i x_2^- (x_1^+)^{n-i-1} v'+\sum_{i=0}^{n-1}(x_1^+)^i x_2^-(x_1^+)^{n-i-1}v''_i\in \I^{W,*}(\mathcal{U}),
$$
where 
$v':= 2 \sigma 1_W + 2\rho w_1^+\in W$ and $v''_i:=4\tau_i 1_W+4\xi_iw_1^+\in W$ are not all zero.
Therefore, by proceeding as in Case~1, we reach the desired conclusion.

\vspace{1ex}
		
\noindent\textbf{Case 3:} Suppose hold \eqref{eq: caso 1}. Since
\begin{align*}
        & a_{p,q}^{(i)}(x_1^+,x_2^+,x^-)=(x_1^+)^i\underbrace{\overline{x_1^+}\cdots \widetilde{x_1^+}}_{p}x^-\underbrace{\overline{x_2^+}\cdots \widetilde{x_2^+}}_{p}(x_1^+)^{q-i},\\
        & a_{p,q, w_1^+}^{(i)}(x_1^+,x_2^+,x^-)=(x_1^+)^i\underbrace{\overline{x_1^+}\cdots \widetilde{x_1^+}}_{p}x^-\underbrace{\overline{x_2^+}\cdots \widetilde{x_2^+}}_{p}(x_1^+)^{q-i}w_1^+,
		\end{align*}
substituting $x^-$ by $[x_1^+,x_2^+]$ in relation \eqref{eq: caso 1}, we obtain relation \eqref{eq: case 3}. Hence, by Case~1, we get the desired conclusion.

\vspace{1ex}
		
\noindent\textbf{Case 4:} Suppose hold \eqref{eq: case 2}. Since $b^{(i)}(x^+,x^-)=(x^+)^ix^-(x^+)^{n-i-1}$ and $ b_{w_1^+}^{(i)}(x^+,x^-)=(x^+)^ix^-(x^+)^{n-i-1}w_1^+$, as in the previous case by substituting $x^-$ with $[x_1^+,x_2^+]$ in relation \eqref{eq: case 2}, we get relation \eqref{eq: case 3}. Therefore, by Case~1, we obtain the desired conclusion.

\vspace{1ex}
		
\noindent\textbf{Case 5:} Suppose hold \eqref{eq: case 4}. Since 
\begin{align*}
		   h_{p,q}^{(i)}(x_1^+,x_2^+,x_3^+)&=(x_1^+)^i\underbrace{\widehat{x_1^+}\cdots \widetilde{x_1^+}}_{p-1}\overline{x_1^+}\,\overline{x_2^+} \,\overline{x_3^+}\underbrace{\widehat{x_2^+}\cdots \widetilde{x_2^+}}_{p-1}(x_1^+)^{q-i},\\
		h_{p,q,w_1^+}^{(i)}(x_1^+,x_2^+,x_3^+)&=(x_1^+)^i\underbrace{\widehat{x_1^+}\cdots \widetilde{x_1^+}}_{p-1}\overline{x_1^+}\,\overline{x_2^+} \,\overline{x_3^+}\underbrace{\widehat{x_2^+}\cdots \widetilde{x_2^+}}_{p-1}(x_1^+)^{q-i}w_1^+, 
		\end{align*}
by substituting $x_3^+$ with $(x_1^+)^2$ in \eqref{eq: case 4} we get \eqref{eq: case 3}, and so by Case~1 we reach the desired conclusion.
\end{proof}

\begin{lemma}\label{lem: subvariety identity}
    Let $\mathcal{U}$ be a proper subvariety of $\V^{W,*}(M^D)$. Then there exist constants $K'\leq N'$ such that
    \begin{equation*}
        (x_1^+)^{K'}x_2^{\epsilon}(x_1^+)^{N'-K'}+\sum_{i<K'} \delta_i(x_1^+)^ix_2^{\epsilon}(x_1^+)^{N'-i}+\sum_{i<K'} \eta_i(x_1^+)^ix_2^{\epsilon}(x_1^+)^{N'-i}w_1^+\in \I^{W,*}(\mathcal{U}),
    \end{equation*}
    for some $\delta_i, \eta_i\in F$ and for any $\epsilon\in\{+,-\}$.
\end{lemma}
\begin{proof}
     By Lemma \ref{lemma: M,N-M}, there exist constants $K\leq N$ such that
    $$
    (x_1^+)^{K}x_2^-(x_1^+)^{N-K}+\hat{\alpha}(x_1^+)^Kx_2^-(x_1^+)^{N-K} w_1^++\sum_{i<K} \hat{\beta}_i(x_1^+)^ix_2^-(x_1^+)^{N-i} + \sum_{i< K}  \hat{\gamma}_i(x_1^+)^ix_2^-(x_1^+)^{N-i} w_1^+ \in \I^{W,*}(\mathcal{U}),
    $$
    where $\hat{\alpha},\hat{\beta}_i, \hat{\gamma}_i\in F$.
    Since $[x_1^+,w_1^+]$ is skew-symmetric and $x_1^-x_2^-\in \I^{W,*}(M^D)\subset \I^{W,*}(\mathcal{U}) $, we can move $w_1^+$ and obtain
    \begin{equation}\label{eq: id K, N-K}
        (x_1^+)^{K}(x_2^-+\hat{\alpha}x_2^-w_1^+)(x_1^+)^{N-K}+\sum_{i<K} (x_1^+)^i(\hat{\beta}_ix_2^-+\hat{\gamma}_ix_2^-w_1^+)(x_1^+)^{N-i} \in \I^{W,*}(\mathcal{U}).
    \end{equation}

First, observe that we may assume, without loss of generality, that $N=2K$ in \eqref{eq: id K, N-K}. Indeed, if $N>2K$, we multiply \eqref{eq: id K, N-K} on the left by  $(x_1^+)^{N-2K}$, thereby reducing to the case $N=2K$. Similarly, if $N<2k$,  we multiply \eqref{eq: id K, N-K} on the right by $(x_1^+)^{2K-N}$, and again we reduce to the case $N=2K$. Then,  \eqref{eq: id K, N-K} can be written as
    $$
    f:=(x_1^+)^{K}(x_2^-+\hat{\alpha}x_2^-w_1^+)(x_1^+)^{K}+\sum_{i<K} (x_1^+)^i(\hat{\beta}_ix_2^-+\hat{\gamma}_ix_2^-w_1^+)(x_1^+)^{2K-i} \in \I^{W,*}(\mathcal{U}).
    $$
    Since $f^*\in \I^{W,*}(\mathcal{U})$ and $(x_2^-w_1^+)^*- x_2^-w_1^+\in \I^{W,*}(M^D)\subset \I^{W,*}(\mathcal{U})$, we have
    $$
    g:=(x_1^+)^{K}(-x_2^-+\hat{\alpha}x_2^-w_1^+)(x_1^+)^{K}+\sum_{i<K} (x_1^+)^i(-\hat{\beta}_{2K-i}x_2^-+\hat{\gamma}_{2K-i}x_2^-w_1^+)(x_1^+)^{2K-i} \in \I^{W,*}(\mathcal{U}).
    $$
    Therefore,
    $$
    f-g=2(x_1^+)^{K}x_2^-(x_1^+)^{K}+\sum_{i<K} (x_1^+)^i((\hat{\beta}_i+ \hat{\beta}_{2K-i})x_2^-+(\hat{\gamma}_i-\hat{\gamma}_{2K-i})x_2^-w_1^+)(x_1^+)^{2K-i}\in \I^{W,*}(\mathcal{U}).
    $$
    As a consequence, we obtain
    $$(x_1^+)^{K}x_2^-(x_1^+)^{K}+\sum_{i<K} \beta'_i(x_1^+)^ix_2^-(x_1^+)^{2K-i}+\sum_{i<K} \gamma'_i(x_1^+)^ix_2^-(x_1^+)^{2K-i}w_1^+\in \I^{W,*}(\mathcal{U}),
    $$
   where $\beta'_i=2^{-1}(\hat{\beta}_i+ \hat{\beta}_{2K-i})$ and $\gamma'_i=2^{-1}(\hat{\gamma}_i-\hat{\gamma}_{2K-i})$. Now, if we substitute $x_2^-$ by $[x_1^+,x_2^+]$, we obtain
$$(x_1^+)^{K+1}x_2^+(x_1^+)^{K}+\sum_{i<K+1} \beta''_i(x_1^+)^ix_2^+(x_1^+)^{2K+1-i}+\sum_{i<K+1} \gamma''_i(x_1^+)^ix_2^+(x_1^+)^{2K+1-i}w_1^+\in \I^{W,*}(\mathcal{U}),
    $$
for some $\beta''_i, \gamma''_i\in F$.
 Hence, if we set $K'=K+1$ and $N'=2K+1$, then it follows that  
$$(x_1^+)^{K'}x_2^{\epsilon}(x_1^+)^{N'-K'}+\sum_{i<K'} \delta_i(x_1^+)^ix_2^{\epsilon}(x_1^+)^{N'-i}+\sum_{i<K'} \eta_i(x_1^+)^ix_2^{\epsilon}(x_1^+)^{N'-i}w_1^+\in \I^{W,*}(\mathcal{U}),
    $$
 for some $\delta_i, \eta_i\in F$, where $\epsilon\in\{+,-\}$, as desired. 
  %
%
\end{proof}

\begin{proposition}\label{prop: n barra}
Let $\mathcal{U}$ be a proper subvariety of $\V^{W,*}(M^D)$. Then there exists a constant $\bar{N}$ such that $m_{\lambda,\mu}^{\mathcal{U}}\leq \bar{N}$ for any $(\lambda,\mu)\vdash(r,n-r)$, $0\leq r \leq n$, $n\geq 1$.
\end{proposition}
\begin{proof}
 By Lemma \ref{lem: subvariety identity}, there exists a constant $N'$ such that
    \begin{equation}\label{eq: N-M}     (x_1^+)^{K'}x_2^{\epsilon}(x_1^+)^{N'-K'}+\sum_{i<K'} \delta_i(x_1^+)^ix_2^{\epsilon}(x_1^+)^{N'-i}+\sum_{i<K'} \eta_i(x_1^+)^ix_2^{\epsilon}(x_1^+)^{N'-i}w_1^+\in \I^{W,*}(\mathcal{U}),
\end{equation}
for some $\delta_i, \eta_i \in F$ and a suitable $K'\leq N'$, where $\epsilon\in\{+,-\}$. We will prove that $m_{\lambda, \mu}^{\mathcal{U}}\leq 4N'$  for any $(\lambda,\mu)\vdash(r,n-r)$, $0\leq r \leq n$, $n\geq 1$.

By Theorem \ref{th: cocharacter}, we only have to consider the following cases:
\begin{itemize}
\item for $r=n$,
$$
(\lambda,\mu)=
\begin{cases}
        ((p+q,p), \emptyset)\\
        ((p+q,p,1),\emptyset)\\
        ((n),\emptyset)
    \end{cases};
$$

\vspace{1mm}

\item for $r=n-1$,
$$
(\lambda,\mu)=\begin{cases}
             ((p+q,p),(1))\\
              ((n-1),(1))\\
    \end{cases}.
$$
\end{itemize}
We proceed with the same reasoning as in \cite[Proposition 5.6]{MarRiz} and \cite[Theorem 3]{MisVal}.

We prove the statement for $\lambda=(p+q,p)$ and $\mu=\emptyset$. The other cases will follow with similar arguments.
If $q \leq  N'-1 $,  then by Theorem \ref{th: cocharacter} we have $m_{\lambda,\mu}^{\mathcal{U}}\leq 4N'$, as desired. So, let us assume that $q\geq N'$ and consider $(W,*)$-polynomials  $f_{p,q}^{(i)}(x_1^+,x_2^+)$, $f_{p,q,w_1^+}^{(i)}(x_1^+,x_2^+)$, $g_{p,q}^{(i)}(x_1^+,x_2^+)$ and $g_{p,q,w_1^+}^{(i)}(x_1^+,x_2^+)$, $0\leq  i\leq q$, defined in \eqref{eq: f}, \eqref{eq: fw_1}, \eqref{eq: gi} and \eqref{eq: gw_1}, respectively, of Lemma \ref{lemma: cn,0}. 
Since
$\underbrace{\overline{x_1^+}\cdots \widetilde{x_1^+}}_{p-1}[x_1^+,x_2^+]\underbrace{\overline{x_2^+}\cdots \widetilde{x_2^+}}_{p-1}$ is a symmetric or a skew-symmetric element according as $p$ is even or odd, respectively, from relation \eqref{eq: N-M} it follows that
\begin{equation*}\label{eq: f polinomi}
    \begin{split}
(x_1^+)^{K'}\underbrace{\overline{x_1^+}\cdots \widetilde{x_1^+}}_{p-1}[x_1^+,x_2^+]\underbrace{\overline{x_2^+}\cdots \widetilde{x_2^+}}_{p-1}(x_1^+)^{N'-K'}\equiv & 
\sum_{i<K'} \delta'_i (x_1^+)^{i}\underbrace{\overline{x_1^+}\cdots \widetilde{x_1^+}}_{p-1}[x_1^+,x_2^+]\underbrace{\overline{x_2^+}\cdots \widetilde{x_2^+}}_{p-1}(x_1^+)^{N'-i}  \\
&+\sum_{i < K'} \eta'_i(x_1^+)^i\underbrace{\overline{x_1^+}\cdots \widetilde{x_1^+}}_{p-1}[x_1^+,x_2^+]\underbrace{\overline{x_2^+}\cdots \widetilde{x_2^+}}_{p-1}(x_1^+)^{N'-i}w_1^+  \pmod{\I^{W,*}(\mathcal{U})},
\end{split}
\end{equation*}
and by acting on the right by $w_1^+$,
\begin{equation*}\label{eq: f polinomi2}
    \begin{split}
(x_1^+)^{K'}\underbrace{\overline{x_1^+}\cdots \widetilde{x_1^+}}_{p-1}[x_1^+,x_2^+]\underbrace{\overline{x_2^+}\cdots \widetilde{x_2^+}}_{p-1}(x_1^+)^{N'-K'}w_1^+\equiv & 
\sum_{i<K'}\delta'_i (x_1^+)^{i}\underbrace{\overline{x_1^+}\cdots \widetilde{x_1^+}}_{p-1}[x_1^+,x_2^+]\underbrace{\overline{x_2^+}\cdots \widetilde{x_2^+}}_{p-1}(x_1^+)^{N'-i}w_1^+  \\
&+\sum_{i< K'} \eta'_i(x_1^+)^i\underbrace{\overline{x_1^+}\cdots \widetilde{x_1^+}}_{p-1}[x_1^+,x_2^+]\underbrace{\overline{x_2^+}\cdots \widetilde{x_2^+}}_{p-1}(x_1^+)^{N'-i} \pmod{\I^{W,*}(\mathcal{U})},
\end{split}
\end{equation*}
where  $\delta'_i, \eta'_i\in F$.
Hence, since $q\geq N'$, we can apply Relation \eqref{eq: N-M} to any polynomials $f_{p,q}^{(i)}(x_1^+,x_2^+)$ and $f_{p,q,w_1^+}^{(i)}(x_1^+,x_2^+)$ such that $i\geq K'$ and, as a consequence, we get
$$
f_{p,q}^{(i)}(x_1^+,x_2^+)\equiv \sum_{j<K'} \delta'_j f_{p,q}^{(j)}(x_1^+,x_2^+) + \sum_{j< K'} \eta'_j f_{p,q,w_1^+}^{(j)}(x_1^+,x_2^+)  \pmod{\I^{W,*}(\mathcal{U})}
$$
and
$$
f_{p,q, w_1^+}^{(i)}(x_1^+,x_2^+)\equiv \sum_{j<K'} \delta'_j f_{p,q,w_1^+}^{(j)}(x_1^+,x_2^+) + \sum_{j< K'} \eta'_j f_{p,q}^{(j)}(x_1^+,x_2^+)  \pmod{\I^{W,*}(\mathcal{U})}.
$$
Similarly, since the polynomial $\underbrace{\overline{x_1^+}\cdots \widetilde{x_1^+}}_{p-1}\left(x_1^+w_1^+x_2^+-x_2^+w_1^+x_1^+\right)\underbrace{\overline{x_2^+}\cdots \widetilde{x_2^+}}_{p-1}$ is symmetric when $p$ is even and skew-symmetric when $p$ is odd and $q\geq N'$, the relation \eqref{eq: N-M} can be applied to the polynomials $g_{p,q}^{(i)}(x_1^+,x_2^+)$ and $g_{p,q,w_1^+}^{(i)}(x_1^+,x_2^+)$ for every $i\geq K'$. Consequently,
$$
g_{p,q}^{(i)}(x_1^+,x_2^+)\equiv \sum_{j<K'} \delta'_j g_{p,q}^{(j)}(x_1^+,x_2^+) + \sum_{j< K'} \eta'_j g_{p,q,w_1^+}^{(j)}(x_1^+,x_2^+)  \pmod{\I^{W,*}(\mathcal{U})}
$$
and 
$$
g_{p,q, w_1^+}^{(i)}(x_1^+,x_2^+)\equiv \sum_{j<K'} \delta'_j g_{p,q,w_1^+}^{(j)}(x_1^+,x_2^+) + \sum_{j< K'} \eta'_j g_{p,q}^{(j)}(x_1^+,x_2^+)  \pmod{\I^{W,*}(\mathcal{U})}.
$$
 Therefore, it follows that $m_{\lambda,\mu}^{\mathcal{U}}\leq 4K' \leq  4N'$, as desired.
\end{proof}

Now, we can finally prove the main theorem of this section.

\begin{theorem}\label{th: MDapg}
    The $(W,*)$-variety generated by $M^D$ has almost polynomial growth.
\end{theorem}
\begin{proof}
    Let $\mathcal{U}$ be a proper subvariety of $\mathcal{V}=\V^{W,*}(M^D)$. We shall prove that $\mathcal{U}$ has polynomial growth of the $(W,*)$-codimensions. By Lemma \ref{lem: subvariety identity}, there exist constants $K'\leq N'$ such that
    \begin{equation*}
         (x_1^+)^{K'}x_2^-(x_1^+)^{N'-K'}+\sum_{i<K'} \delta_i(x_1^+)^ix_2^-(x_1^+)^{N'-i}+\sum_{i<K'} \eta_i(x_1^+)^ix_2^-(x_1^+)^{N'-i}w_1^+\in \I^{W,*}(\mathcal{U}),
    \end{equation*}
    for some $\delta_i, \eta_i\in F$. By substituting $x_2^-$ with $x_{N'+1}^-$ and applying the standard multilinearization process (see \cite[Theorem 1.3.8]{GiaZai}), we get
\begin{align*}
       \sum_{\sigma \in S_{N'}}x_{\sigma(1)}^+\cdots x_{\sigma(K')}^+x_{N'+1}^-x_{\sigma(K'+1)}^+\cdots x_{\sigma(N')}^+
        +\sum_{i<K'}\sum_{\sigma \in S_{N'}} \delta_i x_{\sigma(1)}^+\cdots x_{\sigma(i)}^+x_{N'+1}^-x_{\sigma(i+1)}^+\cdots x_{\sigma(N')}^+ &\\
       +\sum_{i<K'}\sum_{\sigma \in S_{N'}}\eta_i x_{\sigma(1)}^+\cdots x_{\sigma(i)}^+x_{N'+1}^-x_{\sigma(i+1)}^+\cdots x_{\sigma(N')}^+ w_1^+ & \in \Id^{W,*}(\mathcal{U}).
    \end{align*}
Now, we multiply on the right by $y_1^+ \cdots y_{K'}^+$ where $y_i^+$ are new symmetric variables, and alternate $x_i^+$ with $y_i^+$ for all $1\leq i \leq K'$.  Since $[x_1^+,w_1^+]$ is skew-symmetric,
    $x_1^- x_2^- \in \I^{W,*}(M^D)\subset\I^{W,*}(\mathcal{U})$ and by Lemma \ref{Lem: consequence x_1^- x_2^-}, we can move $w_1^+$ to the right of any monomial in which it appears. Moreover, 
    by Lemma \ref{Lem: consequence x_1^- x_2^-} (see \cite[Remark 2]{MisVal}), any symmetric variables to the right (and to the left) of $x_{N'+1}^-$ can be reordered modulo $\I^{W,*}(M^D)\subset\I^{W,*}(\mathcal{U})$. Hence, we get
    $$
     \overline{x_1^+}\cdots \widetilde{x_{K'}^+}x_{N'+1}^-\overline{y_1^+}\cdots \widetilde{y_{K'}^+}{x_{K'+1}^+}\cdots {x_{N'}^+}\ \in \Id^{W,*}(\mathcal{U}).
    $$
    If we multiply on the left by $y_{K'+1}^+\cdots y_{N'}^+$ and we alternate $x_j^+$ with $y_j^+$ for all $K'+1\leq j\leq N'$, we obtain
 \begin{equation}\label{eq: x- id}
     \overline{x_1^+}\cdots \widetilde{x_{N'}^+}x_{N'+1}^-\overline{y_1^+}\cdots \widetilde{y_{N'}^+}\ \in \Id^{W,*}(\mathcal{U}),
 \end{equation}
 and by acting on the left by $w_1^+$ we get
  \begin{equation}\label{eq: x- id w_1}
     \overline{x_1^+}\cdots \widetilde{x_{N'}^+}x_{N'+1}^-\overline{y_1^+}\cdots \widetilde{y_{N'}^+}w_1^+\ \in \Id^{W,*}(\mathcal{U}).
 \end{equation}
This proves that, for all $q\geq 0$ and $0\leq i\leq q$, the $(W,*)$-polynomials $a_{p,q}^{(i)}$ and $a_{p,q,w_1^+}^{(i)}$, defined in \eqref{eq: a} and \eqref{eq: aw_1}, respectively, in Lemma~\ref{lemma: cn-1,1}, belong to $\I^{W,*}(\mathcal{U})$ whenever $p\geq N'$. Consequently, $m_{(q+p,p),(1)}^{\mathcal{U}}=0$ when $p\geq N'$.

    Since $z_1^+w_1^+z_2^+-z_2^+w_1^+z_1^+$ is skew-symmetric,  if we substitute  $x_{N'+1}^-$ in \eqref{eq: x- id} and \eqref{eq: x- id w_1} with $[z_1^+,z_2^+]$ or with $z_1^+w_1^+z_2^+-z_2^+w_1^+z_1^+$, where $z_i^+$ are new symmetric variables, we obtain
$$
         \overline{x_1^+}\cdots \widetilde{x_{N'}^+}[z_1^+,z_2^+]\overline{y_1^+}\cdots \widetilde{y_{N'}^+},\quad \overline{x_1^+}\cdots \widetilde{x_{N'}^+}[z_1^+,z_2^+]\overline{y_1^+}\cdots \widetilde{y_{N'}^+}w_1^+\ \in \Id^{W,*}(\mathcal{U}),
         $$
    and
    $$
         \overline{x_1^+}\cdots \widetilde{x_{N'}^+}(z_1^+w_1^+z_2^+-z_2^+w_1^+z_1^+)\overline{y_1^+}\cdots \widetilde{y_{N'}^+},\quad \overline{x_1^+}\cdots \widetilde{x_{N'}^+}(z_1^+w_1^+z_2^+-z_2^+w_1^+z_1^+)\overline{y_1^+}\cdots \widetilde{y_{N'}^+}w_1^+\ \in \Id^{W,*}(\mathcal{U}).
$$
As above, for all $q\geq 0$ and $0\leq i\leq q$, the $(W,*)$-polynomials $f_{p,q}^{(i)}$, $f_{p,q,w_1^+}^{(i)}$, $g_{p,q}^{(i)}$, and $g_{p,q,w_1^+}^{(i)}$, defined in \eqref{eq: f}--\eqref{eq: gw_1}, respectively, in Lemma~\ref{lemma: cn,0}, belong to $\I^{W,*}(\mathcal{U})$ whenever $p\geq N'+1$. Consequently, $m_{(q+p,p),\emptyset}^{\mathcal{U}}=0$ when $p\geq N'+1$.

    If we substitute $x_{N'+1}^-$ in \eqref{eq: x- id} and \eqref{eq: x- id w_1}  with $\widehat{z_1^+}\widehat{z_2^+}\widehat{z_3^+}$, where $z_i^+$ are new symmetric variables, we obtain 
$$
 \overline{x_1^+}\cdots \widetilde{x_{N'}^+}\widehat{z_1^+}\widehat{z_2^+}\widehat{z_3^+}\overline{y_1^+}\cdots \widetilde{y_{N'}^+},\quad \overline{x_1^+}\cdots \widetilde{x_{N'}^+}\widehat{z_1^+}\widehat{z_2^+}\widehat{z_3^+}\overline{y_1^+}\cdots \widetilde{y_{N'}^+}w_1^+\ \in \Id^{W,*}(\mathcal{U}).$$
Thus, as in the previous cases, for all $q\geq 0$ and $0\leq i\leq q$, the $(W,*)$-polynomials $h^{(i)}_{p,q}$ and $h_{p,q,w_1^+}^{(i)}$, defined in \eqref{eq: h} and \eqref{eq: hw_1}, respectively, in Lemma~\ref{lemma: cn,0}, belong to $\I^{W,*}(\mathcal{U})$ whenever $p\geq N'+1$. Consequently, $m_{(q+p,p,1),\emptyset}^{\mathcal{U}}=0$ when $p\geq N'+1$.

Finally, since $\mathcal{U}$ is a subvariety of $\V^{W,*}(M^D)$, by Theorem \ref{th: cocharacter} we have that $m_{\lambda,\mu}^{\mathcal{U}}=0$ for all $(\lambda,\mu)\vdash (r,n-r)$ and $0\leq r \leq n-2$.

In conclusion, by Theorem \ref{th: cocharacter} in all cases we have proved that $m_{\lambda,\mu}^{\mathcal{U}}=0$ for all $\lambda$ with more than $N'+2$ boxes below the first row
or, equivalently, if $\chi_\lambda \otimes \chi_\mu$ appears with a non-zero multiplicity in the generalized $S_r\times S_{n-r}$-cocharacter of $\mathcal U$, then $\lambda$ must contain at most $N'+1$ boxes below the first row. As a consequence, for $0\leq r \leq n$ and $n\geq 0$ we have
\begin{equation}\label{eq: chi U}
      \chi_{r,n-r}^{W,*}(\mathcal{U})=\sum_{\substack{(\lambda,\mu)\vdash (r,n-r)\\
    r-\lambda_1\leq N'+1\\
    n-r\leq 1}}m_{\lambda, \mu}^{\mathcal{U}}(\chi_{\lambda} \otimes \chi_{\mu}),
\end{equation}
    where $\lambda_1$ stands for the number of boxes of the first row of $\lambda.$

    Since $r-\lambda_1\leq N'+1$, then $\lambda_1\geq r-(N'+1)$. By the hook formula 
    $$
    d_{\lambda}=\chi_{\lambda}(1)\leq \frac{r!}{(r-(N'+1))!}\leq r^{N'+1}\leq n^{N'+1}.
    $$ 
   Also recall that $d_{\mu}=\chi_{\mu}(1)=1$ whenever $\mu=\emptyset$ or $\mu=(1)$. 
   Thus, since $c_{r,n-r}^{W,*}(\mathcal{U})= \chi_{r,n-r}^{W,*}(\mathcal{U})(1)$, by Formulas~\eqref{eq: decomp codim} and \eqref{eq: chi U}, and Proposition~\ref{prop: n barra}, we get
   \begin{align*}
       c_n^{W,*}(\mathcal{U})=& \sum_{r=0}^n \binom{n}{r} \chi_{r,n-r}^{W,*}(\mathcal{U})(1)= \sum_{r=0}^n \binom{n}{r} \sum_{\substack{(\lambda,\mu)\vdash (r,n-r)\\
    r-\lambda_1\leq N'+1\\
    n-r\leq 1}}m_{\lambda, \mu}^{\mathcal{U}} d_{\lambda}d_{\mu}\\
    =& \sum_{\substack{(\lambda,\mu)\vdash (n,0)\\
    n-\lambda_1\leq N'+1}}m_{\lambda, \mu}^{\mathcal{U}} d_{\lambda}d_{\mu} +n \sum_{\substack{(\lambda,\mu)\vdash (n-1,1)\\
    n-1-\lambda_1\leq N'+1}}m_{\lambda, \mu}^{\mathcal{U}} d_{\lambda}d_{\mu}\\
    \leq & \sum_{\substack{(\lambda,\mu)\vdash (n,0)\\
    n-\lambda_1\leq N'+1}}\bar{N}n^{N'+1}+n \sum_{\substack{(\lambda,\mu)\vdash (n-1,1)\\
    n-1-\lambda_1\leq N'+1}}\bar{N}n^{N'+1} \leq (N'+1)^2\bar{N}(n+1)n^{N'+1},
   \end{align*}
    since the number of partitions such that $r-\lambda_1\leq N'+1$ is bounded by $(N'+1)^2$.
    Therefore, $c^{W,*}_n(\mathcal{U})$ is polynomially bounded and $\V^{W,*}(M^D)$ has almost polynomial growth, as required.
\end{proof}

As a direct consequence of Corollary \ref{cor: mf md distinct}, Theorem \ref{th: MFapg} and Theorem \ref{th: MDapg}, we have the following result.

\begin{corollary}
  The $(W,*)$-algebras $M^F$ and $M^D$ generate two distinct $(W,*)$-varieties of almost polynomial growth.
\end{corollary}

Notice that the same cannot be said for $\V^{W,*}(M^{C_i})$, for $1\leq i \leq 3$,  and $\V^{W,*}(M)$ since  by Theorems \ref{th: WM for W*A} and \ref{th: ss} they contain either $M^F$ or $M^D$. Thus, we conclude with the following result.

\begin{corollary}\label{cor: M APG}
    The $(W,*)$-variety $\V^{W,*}(M^{\Phi})$ has almost polynomial growth if and only if $\Phi(W)$ is either isomorphic to $F$ or $ D$.
\end{corollary}

\end{document}